\documentclass{amsart}

\usepackage[letterpaper,top=2cm,bottom=2cm,left=3cm,right=3cm,marginparwidth=1.75cm]{geometry}

\usepackage[utf8]{inputenc}
\usepackage{amsmath}
\usepackage{mathtools}   % loads »amsmath«
\usepackage{amsfonts}
\usepackage{listings}
\usepackage{cancel}
\usepackage{comment}
\usepackage{mathabx}
\usepackage{amssymb}

\usepackage{hyperref}
\usepackage{verbatim}

\theoremstyle{plain}   						
\newtheorem{theorem}{Theorem}[section]			
\newtheorem{lemma}[theorem]{Lemma} 		 	
\newtheorem{proposition}[theorem]{Proposition}	
\newtheorem{corollary}[theorem]{Corollary}

\theoremstyle{remark} 
\newtheorem{remark}{Remark}[section]

\theoremstyle{definition}
\newtheorem{definition}{Definition}[section]

\title{Instability of gray solitons in a Gross-Pitaevskii model with a moving impurity}
\author[]{Paolo Antonelli}
\address{Gran Sasso Science Institute, Viale Francesco Crispi, 7, 67100, L'Aquila, Italy}
\email{paolo.antonelli@gssi.it} 
\author[]{Martino Caliaro}
\address{Gran Sasso Science Institute, Viale Francesco Crispi, 7, 67100, L'Aquila, Italy}
\email{martino.caliaro@gssi.it}
\keywords{Gross-Pitaevskii equation, external potential, traveling waves, Evans function, spectral stability}
\subjclass{Primary: 35Q55, 35B35,. Secondary: 35B32, 37K45, 37K50.}

\begin{document}
\begin{abstract}
    The effect of a moving impurity in a dilute Bose-Einstein condensate is investigated by means of the one-dimensional Gross-Pitaevskii model (GP) with non-zero boundary conditions at infinity. The impurity is modeled as a localized external potential, that travels at constant speed $v \in \mathbf{R}$. In a co-moving reference frame, we study the existence and stability of time-independent solutions. The latter are of physical relevance, being associated with the superfluid behavior of the condensate. \\
    For every non-zero velocity $v$ in the subsonic regime, we show the existence of a family of time-independent solutions which bifurcates from a (displaced) gray soliton $\phi_{0,v}(x-s_0)$, with $s_0 \in \mathbf{R}$, of the GP equation. The position $s_0$ is determined as an extremal point of an \textit{effective potential} explicitly defined. Moreover, we study the spectral stability of these states. For small values of the potential strength, we show that the families originating from the maxima of the effective potential are spectrally unstable. For this last result, we employ an Evans function approach. Finally, we formally apply the instability result to the case of a repulsive delta potential.
\end{abstract}
\maketitle
\section{Introduction}
We consider the one dimensional Gross-Pitaevskii equation in the presence of an external potential $V$ moving with velocity $v \in \mathbf{R}$
\begin{equation}
    i\partial_t u -\frac{1}{2}\partial_x^2u -(1-|u|^2)u + \varepsilon V(x-vt)u = 0, \qquad x \in \mathbf{R}, \quad t \in \mathbf{R}.
    \label{eq: GP_V}
\end{equation}
Here $\varepsilon \in \mathbf{R}$ is the intensity of the potential $V$, and $V: \mathbf{R} \to \mathbf{R}$ is assumed to be a smooth function that decays exponentially fast at infinity. The field $u: \mathbf{R}_x\times \mathbf{R}_t \to \mathbf{C}$ is required to satisfy the asymptotic condition
\begin{equation}
    |u(x)|^2 \to 1 \qquad \text{as} \quad |x| \to \infty.
    \label{eq: condition_2}
\end{equation}
Equation \eqref{eq: GP_V} with condition \eqref{eq: condition_2} describes the motion of an impurity (modeled by V) in a quantum fluid, such as dilute Bose-Einstein condensates (BEC) or superfluid Helium, which is at rest at infinity with constant density. It has received a lot of attention from the physical community (see for example \cite{pham_2, pitaevskii, hakim1997nonlinear, Leboeuf, pavloff, pham-brachet, pham_3, saqlain_kevrekidis}) for its ability to describe certain fundamental phenomena associated with these fluids, such as the one of superfluidity \cite{primer_quantum_fluids, pitaevskii_stringari}.\\
If written in a co-moving reference frame, i.e. after the change of coordinates $x \to x-vt$, equation \eqref{eq: GP_V} reads
\begin{equation}
\tag{GP-V}
    i\partial_t u -\frac{1}{2}\partial_x^2u + iv\partial_xu -(1-|u|^2)u + \varepsilon V(x)u = 0.
    \label{eq: GP_V_main}
\end{equation}
This form of the equation, together with condition \eqref{eq: condition_2}, will be the subject of our studies.\\
Most of the previous research on \eqref{eq: GP_V_main} has focused on the existence of time independent solutions (also called steady flows). As we discuss later on, the interest in these solutions is motivated by their relation with the phenomenon of superfluidity. The latter, in the present setting, corresponds to a frictionless motion of the impurity through the quantum fluid (see \cite{Leboeuf, pavloff}). \\
When $\varepsilon=0$, the set of time-independent solutions to \eqref{eq: GP_V_main} is well known \cite{bethuel_existence}. For any $v \in (-1,1)$, there exists only two time-independent solutions, up to phase shifts and translations. These are given by the constant $u = 1$ and by the so called \textit{dark} soliton
\begin{equation}
    \phi_{0,v}(x) = \sqrt{1-v^2}\tanh(\sqrt{1-v^2}x)+iv, \qquad v \in (-1,1).
    \label{eq: stat_sol_intro}
\end{equation}
On the other hand, for $|v| \geq 1$ the only time-independent solution is the constant $u=1$. Dark solitons play a central role in our analysis. They have been extensively studied in the past \cite{bethuel_existence, chiron_paper} and they are usually split into two kinds: the \textit{black} soliton or \textit{kink}, corresponding to the case $v=0$; and the \textit{gray} solitons, corresponding to $v \in (-1,1)\backslash\{0\}$.  As shown by Lin in \cite{zhiwu} (see also \cite{barashenkov,bethuel_existence, chiron_paper}) and by Di Menza and Gallo in  \cite{gallo_dark} (see also \cite{bethuel, pelinovsky}), dark solitons are orbitally stable solutions to \eqref{eq: GP_V_main}.\\
The existence of time-independent solutions to \eqref{eq: GP_V_main} for $\varepsilon \neq 0$ has been studied by Pelinovsky and Kevrekidis in \cite{pelinovsky} for the case $v=0$. By means of a Lyapunov-Schmidt procedure, they show the existence of a family of time-independent solutions to \eqref{eq: GP_V_main}, which arise as regular perturbations of the black soliton $\phi_{0,0}$. Moreover, combining an Evans function approach with a matched asymptotic analysis, they show the linear instability of these states.\\
The case of our interest is the one in which $v \in (-1,1) \backslash\{0\}$. In this regime, the existence of time-independent solutions to \eqref{eq: GP_V_main} with $\varepsilon \neq 0$ has been studied by Hakim in \cite{hakim1997nonlinear} and by Mari\c{s} in \cite{maris2003}. For weak potentials, i.e. for $|\varepsilon|<<1$, Hakim used a formal perturbation argument to show, for each $v \in (-1,1)\backslash\{0\}$ and under a non-degeneracy condition, the existence of two families of time-independent solutions. As $\varepsilon$ varies from zero, one family bifurcates from the gray soliton $\phi_{0,v}$, and its elements we denote by $\phi_{\varepsilon,v}$; the other family bifurcates from the constant solution $u=1$.\\
The stability of these two families of solutions is investigated by Hakim by means of numerical simulations (see also \cite{pham-brachet, pham_3}). While the solutions belonging to the second family, i.e. the one bifurcating from $u=1$, appear to be stable under \eqref{eq: GP_V_main}, the states $\phi_{\varepsilon,v}$ appear to be unstable. The instability manifests in a decrease in the fluid density at the impurity location, followed by the emission of gray solitons in the downstream direction and sound waves in the upstream direction. This behavior is similar to the one observed in the instability of steady flows for the two dimensional version of equation \eqref{eq: GP_V_main}, with gray solitons replaced by quantum vortices \cite{frisch}.\\
The main goal of the present work is to prove the linear instability of the states $\phi_{\varepsilon,v}$, for each $v \in (-1,1) \backslash\{0\}$. The approach we follow consists in linearizing \eqref{eq: GP_V_main} around $\phi_{\varepsilon,v}$ and in studying the corresponding eigenvalue problem $\mathcal{L}_{\varepsilon}u = \lambda u$. The linear instability of $\phi_{\varepsilon,v}$ corresponds to the existence of an eigenvalue $\lambda \in \mathbf{C}$ with positive real part. Differently from the case $v=0$ treated in \cite{pelinovsky}, the matrix operator $\mathcal{L}_{\varepsilon}$ is not off-diagonal. As we will see, this is due to the presence of the drift term $iv\partial_xu$ in \eqref{eq: GP_V_main} and to the complex-valued nature of gray solitons. Because of this structure, the analysis of $\mathcal{L}_{\varepsilon}$ will present additional difficulties with respect to the case $v=0$.\\
The main tool we use in order to locate the eigenvalues of $\mathcal{L}_{\varepsilon}$ and to show the instability of $\phi_{\varepsilon,v}$ is the Evans function $E(\lambda, \varepsilon)$, as we will see in Section \ref{section: unperturbed_problem}. However, in order to fully utilize this tool, it is useful to have more detailed information on the family of solutions $\phi_{\varepsilon,v}$. For this reason, we begin by studying the existence and properties of the family of time-independent solution $\phi_{\varepsilon,v}$, providing in this way a rigorous justification of the formal results in \cite{hakim1997nonlinear}. Then, we proceed to study the instability of these states.

\subsection{\texorpdfstring{Existence and properties of the family $\phi_{\varepsilon,v}$}{Existence and properties of the family phi varepsilon,v}}
Following the ideas in \cite{pelinovsky}, we provide a proof of the existence of the family of time-independent solutions $\phi_{\varepsilon,v}$. Moreover, we show that these solutions share similar properties with gray solitons. These properties define a class of time-independent solutions to \eqref{eq: GP_V_main}. In analogy with the terminology used in \cite{pelinovsky}, we denote the elements of this class as \textit{gray modes}.
\begin{definition}
    Let $\varepsilon \in \mathbf{R}$ and $v \neq 0$. A \textit{gray mode} is time-independent solution $u:\mathbf{R} \to \mathbf{C}$ to \eqref{eq: GP_V_main} which is nowhere vanishing. It can be written as
    \begin{equation}
        u(x) = \rho(x)e^{i\theta(x)}e^{i\phi}, \qquad \phi \in \mathbf{R},
    \end{equation}
    for $\rho: \mathbf{R} \to \mathbf{R}_+$ and $\theta: \mathbf{R} \to \mathbf{R}$ smooth functions of their argument which converge exponentially fast to the asymptotic conditions
    \begin{equation}
        \lim_{|x| \to \infty}\rho(x) = 1 \quad \text{and} \quad \lim_{x \to \pm \infty}\theta(x) = \theta_{\pm}.
    \end{equation}
    Here, $\theta_{\pm} \in \mathbf{R}$ are the parameters of the solution.
    \label{def: traveling_mode}
\end{definition}
Our first main result shows the existence of a family of gray modes $\phi_{\varepsilon,v}$, which, as $\varepsilon$ varies from zero, bifurcates from the gray soliton $\phi_{0,v}(x-s_0)$, with $s_0 \in \mathbf{R}$. The position $s_0$ is determined as a non-degenerate critical point of the effective potential
\begin{equation}
    M(s) := \int_{\mathbf{R}}V(x)[1-|\phi_{0,v}|^2(x-s)]dx, \qquad s \in \mathbf{R},
\end{equation}
and may be not unique.
\begin{proposition}
Let $v \in (-1,1) \backslash \{0\}$ and let $s_0$ be a simple root of
\begin{equation}
    M'(s) = \int_{\mathbf{R}}V'(x)[1-|\phi_{0,v}|^2(x-s)]dx, \qquad s \in \mathbf{R},
    \label{eq: M'_intro}
\end{equation}
namely $M'(s_0) =0$ and $M''(s_0) \neq 0$, where $\phi_{0,v}(x)$ is defined in \eqref{eq: stat_sol_intro}. There exist $\varepsilon_0>0$ and a unique family $\phi_{\varepsilon,v}(x-s_\varepsilon)$ of gray modes of Definition \ref{def: traveling_mode} defined for $\varepsilon \in (-\varepsilon_0,\varepsilon_0)$ for which
\begin{equation}
    \lim_{\varepsilon \to 0} \ \bigl(||\phi_{\varepsilon,v} - \phi_{0,v}||_{L^{\infty}(\mathbf{R})} + |s_{\varepsilon}-s_0|\bigr)  = 0,
\end{equation}
and such that, for any $x \in \mathbf{R}$ fixed, the map $\varepsilon \to \phi_{\varepsilon,v}(x-s_{\varepsilon})$ is smooth in $(-\varepsilon_0,\varepsilon_0)$.
    \label{prop: persistence_intro}
\end{proposition}
As in \cite{pelinovsky}, we prove Proposition \ref{prop: persistence_intro} by means of a Lyapunov-Schmidt reduction method, this time applied to the hydrodynamical form of equation \eqref{eq: GP_V_main} (see Section \ref{section: existence_gray_modes}).\\
The non-degeneracy condition on the function $M'(s)$ in \eqref{eq: M'_intro} coincides with the one obtained by Hakim in \cite[Equation (28)]{hakim1997nonlinear}, and it arises as a Fredholm alternative. We can give to it a physical interpretation in terms of drag forces. Indeed, the quantity in \eqref{eq: M'_intro} is the average value of $V'(x)$ over the fluid wavefunction and it represents the drag force experienced by the impurity when the fluid has the configuration $\phi_{0,v}(\cdot-s)$ \cite[Section 2]{pavloff}. Then, the location $s_0$ from which gray modes bifurcate is the one at which the drag force $M'(s_0)$ vanishes. Notice that, in the particular case of even external potentials, i.e. if $V(x) = V(-x)$, we have that $s_0=0$ is always a root of $M'(s)$.\\
As already mentioned, the existence of time-independent solutions to \eqref{eq: GP_V_main}, such as gray modes, is associated with the phenomenon of superfluidity. 
Indeed, by their steady nature, these solutions conserve in time the linear momentum of the fluid in any region of space. As a consequence, if a quantum fluid occupies one of these configurations, the exchange of linear momentum between the impurity and the fluid is zero. The impurity is thus allowed to travel in absence of drag forces and this is interpreted as superfluid motion (see \cite[Section 2]{pavloff}). In Lemma \ref{lemma: no_drag}, we study the drag force for the case of gray modes and we show its zero value.
\begin{comment}
    Knowing the existence of a smooth family of traveling modes that bifurcates from a gray soliton solution is functional to the study of the stability of these states.
\end{comment}
\subsection{\texorpdfstring{Linear instability of the family $\phi_{\varepsilon,v}$.}{Linear instability of the family phi varepsilon,v.}} After we studied the existence and properties of the gray modes $\phi_{\varepsilon,v}(x-s_{\varepsilon})$, we turn to study their stability. Our goal is to show that they are spectrally unstable states. We do this by linearizing \eqref{eq: GP_V_main} around one of these states and by studying the corresponding eigenvalue problem $\mathcal{L}_{\varepsilon}u=\lambda u$. The eigenvalues of $\mathcal{L}_{\varepsilon}$ determine the stability of the wave. In particular, the instability of $\phi_{\varepsilon,v}(x-s_{\varepsilon})$ corresponds to the existence of an eigenvalue $\lambda \in \mathbf{C}$ of $\mathcal{L}_{\varepsilon}$ with positive real part. As we describe below, the Evans function is the tool we use to determine the existence of unstable eigenvalues.\\
When $\varepsilon=0$, due to the invariance of \eqref{eq: GP_V_main} under translations and phase shifts, $\lambda=0$ is a multiple eigenvalue of $\mathcal{L}_0$. When $\varepsilon$ acquires a non-zero value the translational symmetry is broken, and the zero eigenvalue associated with this symmetry splits into different eigenvalues. Our goal is to track the motion of these eigenvalues as $\varepsilon$ varies and to show that at least one of them enters the half-plane $\{\lambda \in \mathbf{C}, \ \Re \lambda >0\}$, leading to the instability of the gray mode $\phi_{\varepsilon,v}(x-s_{\varepsilon})$. As we will see, a main mathematical difficulty lies in the fact that the imaginary axis belongs to the essential spectrum of $\mathcal{L}_{\varepsilon}$, for any $\varepsilon$. Thus, the splitting of the zero eigenvalue is associated with an \textit{edge bifurcation}, i.e. to the  bifurcation of eigenvalues out of the essential spectrum \cite{kapitula_dark}.\\
The major tool we use in order to track the location of the eigenvalues of $\mathcal{L}_{\varepsilon}$ is the Evans function $E(\lambda,\varepsilon)$. This is a complex-valued function with the property that $E(\lambda,\varepsilon) =0$ if $\lambda$ is an isolated eigenvalue of $\mathcal{L}_{\varepsilon}$. A priori, the Evans function is defined only away from the essential spectrum of $\mathcal{L}_{\varepsilon}$, and for this reason it cannot be immediately used in order to describe an edge bifurcation. However, as shown by Kapitula and Sandstede in \cite{kapitula&sandstede} and by Gardner and Zumbrun in \cite{gardner_zumbrun} with the Gap Lemma, under some assumptions the Evans function can be extended analytically across the essential spectrum. Using this extension, we are able to study the bifurcation of the zero eigenvalue.\\
In order to compute the zeros of the Evans function we follow the approach by Kapitula and Rubin in \cite{kapitula_dark}, where the persistence and stability of the black soliton $\phi_{0,0}$ under non-hamiltonian perturbations is studied. We begin by defining the Evans function $E(\lambda, \varepsilon)$ associated with $\mathcal{L}_{\varepsilon}$ in a region that excludes the point $\lambda =0$ (and the whole imaginary axis). Then, we extend analytically the Evans function in the vicinity of $\lambda=0$. Here, differently from \cite{kapitula_dark} where $E(\lambda,\varepsilon)$ is extended on a 2-sheeted Riemann surface, the extension is done directly in a neighborhood of zero in $\mathbf{C}$ (see Lemma \ref{lemma: cont_evans}). Finally, we write a Taylor expansion of the Evans function at $\lambda =0$ and $\varepsilon=0$, explicitly computing the coefficients of the leading order terms. We obtain in this way the following proposition (see Section \ref{section: perturbed_case} for a more precise set of statements).
\begin{proposition}
    Let $\phi_{\varepsilon,v}(x-s_{\varepsilon})$ be the gray mode of Proposition \ref{prop: persistence_intro}, that bifurcates from the gray soliton $\phi_{0,v}(x-s_0)$, with $s_0 \in \mathbf{R}$ simple root of $M'(s)$. Let $\mathcal{L}_{\varepsilon}$ be the operator associated with the linearization of \eqref{eq: GP_V_main} around $\phi_{\varepsilon,v}(x-s_{\varepsilon})$, and let $E(\lambda, \varepsilon)$ be the Evans function associated with $\mathcal{L}_{\varepsilon}$. In the vicinity of $\lambda =0$ and $\varepsilon=0$, the Evans function admits the expansion
    \begin{equation}
        E(\lambda,\varepsilon) = -2P'_r(v)\lambda^3-2M''(s_0) \varepsilon \lambda + O(\lambda^4, \lambda^2\varepsilon,\lambda\varepsilon^2),
        \label{eq: Evans_exp_intro}
    \end{equation}
    where  
    \begin{equation}
       P'_r(v) = 4\sqrt{1-v^2} \qquad \text{and} \qquad   M''(s_0)=\int_{\mathbf{R}}V''(x)[1-|\phi_{0,v}(x-s_0)|^2]dx.
       \label{eq: der_renorm_mom_intro}
    \end{equation}
    \label{prop: evans_intro}
\end{proposition}
By setting to zero the expansion in \eqref{eq: Evans_exp_intro} up to first orders, we obtain an approximate value of the zeros of $E(\lambda,\varepsilon)$. We are able in this way to study, as $\varepsilon$ is varied, the bifurcation of the zero eigenvalue of $\mathcal{L}_0$ associated to the translational symmetry.
\begin{corollary}
    The operator $\mathcal{L}_{\varepsilon}$, arising from the linearization of \eqref{eq: GP_V_main} around $\phi_{\varepsilon,v}(x-s_{\varepsilon})$, admits a pair of eigenvalues that satisfy
    \begin{equation}
        \lambda^2 = - \frac{M''(s_0)}{P'_r(v)} \varepsilon + O(\varepsilon^{3/2}).
        \label{eq: unst_eigenvalue_intro}
    \end{equation}
     If, without loss of generality, we assume $\varepsilon>0$, the gray mode $\phi_{\varepsilon,v}(x-s_{\varepsilon})$ is unstable for $M''(s_0)<0$ and $\varepsilon$ small enough.
    \label{prop: instability_intro}
\end{corollary}
\begin{remark}
    The quantity $P'_r(v)$ that appears in Proposition \ref{prop: evans_intro} and in Corollary \ref{prop: instability_intro} is the derivative with respect to $v$ of the \textit{renormalized momentum} of $\phi_{0,v}$. The latter is defined as
\begin{equation}
    P_r(v) := \frac{i}{2}\int_\mathbf{R}\bigl(\overline{\phi}_{0,v}\phi'_{0,v} - \phi_{0,v}\overline{\phi}'_{0,v}\bigr)\Bigl(1-\frac{1}{|\phi_{0,v}|^2}\Bigr)dx, \qquad v \in (-1,1)\backslash\{0\}.
\end{equation}
By explicit computations, one can verify the first identity in \eqref{eq: der_renorm_mom_intro}. The derivative $P'_r(v)$ plays an important role in the stability analysis of gray solitons: it determines the stability criterion for these waves, as proved by Lin \cite{zhiwu} and Barashenkov \cite{barashenkov}. According to this criterion, the stability of gray solitons follows from the condition $P'_r(v)>0$. Moreover, as proven in \cite{pelinovsky}, the condition $P'_r(v)|_{v\downarrow0}>0$ determines the stability of the black soliton $\phi_{0,0}$. Here, the quantity $P'_r(v)$ appears as the coefficient of the leading order term in the expansion \eqref{eq: Evans_exp_intro}.
\label{remark: ren_mom}
\end{remark}
The result in Corollary \ref{prop: instability_intro} implies that, for $\varepsilon>0$ small enough, the gray modes $\phi_{\varepsilon,v}(x-s_{\varepsilon})$ are linearly unstable solutions to \eqref{eq: GP_V_main} whenever $M''(s_0) <0$. As we will discuss, a physically meaningful example of an external potential $V$ for which the instability condition is verified is given by
\begin{equation}
    V_1(x) = \frac{1}{\sqrt{2\pi\sigma^2}}e^{-\frac{x^2}{2\sigma^2}}   \qquad x \in \mathbf{R},
    \label{eq: pot_laser}
\end{equation}
for $\sigma >0$. The physical relevance of $V_1$ is due to the fact that, in many experiments, the effect of a traveling impurity in a BEC is studied by moving a laser beam through the condensate \cite{onofrio, raman}. Then, the potential $V_1$ is often used to model the repulsive effect of the laser on the condensate \cite{hakim1997nonlinear,saqlain_kevrekidis}.\\
We notice that for the case $M''(s_0)>0$, Corollary \ref{prop: instability_intro} does not allow to conclude the stability or instability of the gray modes $\phi_{\varepsilon,v}(x-s_{\varepsilon})$. In order to determine the stable or unstable nature of these waves, it is necessary to compute higher order terms in the expansion \eqref{eq: unst_eigenvalue_intro}. This can be done either by the asymptotic matching approach in \cite{pelinovsky}, or by computing higher order terms in the Taylor expansion of the Evans function $E(\lambda, \varepsilon)$, as in \cite{kapitula_dark}. However, the computations become more involved, and the stability problem for the case $M''(s_0)>0$ is left for future research.\\
Finally we notice that, in the limit $v \to 0$, the result in Corollary \ref{prop: instability_intro} coincides, up to first order, with the one obtained in \cite[Theorem 4.11]{pelinovsky}.\\
\subsection{The case of a short range external potential} The case in which $V$ is given by a repulsive delta potential, i.e. when $\varepsilon V(x) = g \delta(x)$ for $g >0$, and for which \eqref{eq: GP_V_main} reads
\begin{equation}
\tag{GP-$\delta$}
    i\partial_t u -\frac{1}{2}\partial_x^2u + iv\partial_xu -(1-|u|^2)u + g \delta(x)u = 0.
    \label{GP_delta_intro}
\end{equation}
is of particular interest. Indeed, in this case sharp existence conditions and explicit expressions for time-independent solutions are available, as shown by Hakim \cite{hakim1997nonlinear} and Mari\c{s} \cite{maris2003} (see also \cite{Leboeuf,pavloff, pham-brachet, LeCozIR}). Similarly to the case of smooth potentials, time-independent solutions to \eqref{GP_delta_intro} do not exist for velocities $|v| \geq 1$. For any $v \in (-1,1)\backslash\{0\}$, instead, there are two families of time-independent solutions: as $g>0$ varies from zero, one family bifurcates from the gray soliton $\phi_{0,v}$, and its elements we denote by $\phi_{g,v}$; the other family bifurcates from the constant state $u=1$. The numerical simulations performed by Hakim suggest that the time-independent solutions belonging to the second family, i.e. the one bifurcating from $u=1$, are stable under \eqref{GP_delta_intro}; on the other hand, the states $\phi_{g,v}$ appear to be unstable. In our previous work \cite{antonelli_caliaro}, we showed the orbital stability of the states belonging to the second family. Here, we try complement the analysis by showing the linear instability of $\phi_{g,v}$, at least for $g>0$ small enough. The strategy is to \textit{formally} apply the instability result obtained for the case of a smooth potential $V$ to the case of a delta potential. In order to make the argument rigorous, one can consider a sequence of smooth potentials $\{V_n\}_{n \in \mathbf{N}}$ which approximates the delta distribution as $n \to \infty$, and show the convergence of the eigenvalues of the operator $\mathcal{L}_{\varepsilon,n}$ to the ones associated to $\mathcal{L}_g$, as $n \to \infty$. The result we obtain is expressed in the following proposition (we refer to Section \ref{section: delta_potential} for a more complete statement).
\begin{proposition}
    Let $v \in (-1,1)\backslash\{0\}$, and let $\phi_{g,v}$ be the family of time-independent solutions to \eqref{GP_delta_intro}, that bifurcates from the gray soliton $\phi_{0,v}$, obtained in \cite{hakim1997nonlinear}. Let $\mathcal{L}_g$ be the operator associated with the linearization of \eqref{GP_delta_intro} around $\phi_{g,v}$. Then, $\mathcal{L}_{g}$ admits a pair of eigenvalues $\lambda \in \mathbf{C}$ which satisfy
    \begin{equation}
        \lambda^2 = -\frac{M_{\delta}''(0)}{P'_r(v)}g + O(g^{3/2})
        \label{eq: formula_delta}
    \end{equation}
    Here, 
    \begin{equation}
        M_{\delta}''(0) = \int_{\mathbf{R}}\delta(t)\ \partial_t^2\bigl(1-|\phi_{0,v}(t)|^2\bigr)\ dt = -2(1-v^2)^2.
    \end{equation}
    In particular, for any $g>0$ small enough, the states $\phi_{g,v}$ are linearly unstable solutions to \eqref{GP_delta_intro}.
    \end{proposition}

\section{Existence of gray modes}
\label{section: existence_gray_modes}
In this section we study the existence of gray modes for \eqref{eq: GP_V_main}, arising as regular perturbations of the gray solitons in \eqref{eq: stat_sol_intro}. As detailed in Definition \ref{def: traveling_mode}, gray modes are nowhere vanishing solutions to the stationary equation 
\begin{equation}
    -\frac{1}{2}\partial_x^2u + iv\partial_xu -(1-|u|^2)u + \varepsilon V(x)u = 0,
    \label{eq: GP_V_main_stat}
\end{equation}
and admit well-defined limits at $x=\pm \infty$.
Our approach to prove the existence of such solutions relies on the the method of Lyapunov-Schmidt reduction, similarly to \cite{pelinovsky}. However, instead of working directly with equation \eqref{eq: GP_V_main_stat}, we work with its  hydrodynamical formulation. More precisely, we proceed as follows. A nowhere vanishing function $u(x) = \rho(x)e^{i\theta(x)}$, satisfying $\rho(x) \to 1$ and $\partial_x\theta(x) \to 0$ as $|x| \to \infty$, is a solution of \eqref{eq: GP_V_main_stat} if and only if $(\rho,\theta)$ satisfy the following system 
\begin{equation}
    \partial_x\theta(x) = v\Bigl(1-\frac{1}{\rho^2(x)}\Bigr).
    \label{eq: theta}
\end{equation}
and
\begin{equation}
    -\frac{1}{2}\partial_x^2\rho + \frac{v^2}{2}\Bigl(-\rho + \frac{1}{\rho^3}\Bigr) - (1-\rho^2)\rho + \varepsilon V(x)\rho =0.
    \label{eq: rho}
\end{equation}
In order to prove the existence of gray modes for \eqref{eq: GP_V_main}, we apply the Lyapunov-Schmidt reduction method to equation \eqref{eq: rho}. Under a non-degeneracy condition, we obtain a family of solutions $\rho_{\varepsilon,v}$ to \eqref{eq: rho}, satisfying $\rho_{\varepsilon,v}(x) \to 1$ as $|x| \to \infty$, that bifurcates from the state
\begin{equation}
    \rho_{0,v}(x-s_0) = \sqrt{v^2+(1-v^2)\tanh^2\bigl(\sqrt{1-v^2}(x-s_0)\bigr)}, \qquad s_0 \in \mathbf{R}.
    \label{eq: rho_0(x)}
\end{equation}
Upon defining $\theta_{\varepsilon,v}$ as in \eqref{eq: theta}, the function $u_{\varepsilon,v}(x):=\rho_{\varepsilon,v}(x)e^{i\theta_{\varepsilon}(x)}$ defines a gray mode for \eqref{eq: GP_V_main_stat}, bifurcating from the gray soliton $\phi_{0,v}(x-s_0)$.\\ We begin by constructing the family $\rho_{\varepsilon,v}$ in the next proposition.
\begin{proposition}
Let $s_0$ be a simple root of
\begin{equation}
    M'(s) = \int_{\mathbf{R}}V'(x)[1-\rho_{0,v}^2(x-s)]dx, \qquad s \in \mathbf{R},
    \label{eq: M'}
\end{equation}
namely $M'(s_0) =0$ and $M''(s_0) \neq 0$, where $\rho_{0,v}(x)$ is defined in \eqref{eq: rho_0(x)}. There exists $\varepsilon_0>0$ such that, for any $|\varepsilon| < \varepsilon_0$, there is a unique solution $\rho_{\varepsilon,v}(x-s_{\varepsilon})$ to \eqref{eq: rho} such that 
\begin{equation}
    ||\rho_{\varepsilon,v} - \rho_{0,v}||_{L^{\infty}(\mathbf{R})} + |s_{\varepsilon}-s_0| \lesssim \varepsilon.
    \label{eq: order_epsilon}
\end{equation}
Moreover, for any $x \in \mathbf{R}$ fixed, the map $\varepsilon \to \rho_{\varepsilon,v}(x-s_{\varepsilon})$ is smooth in $(-\varepsilon_0,\varepsilon_0)$.
\label{prop: persistence_rho_0}
\end{proposition}
\begin{proof}
Let $v \in (-1,1) \backslash\{0\}$. For $s \in \mathbf{R}$, we use the decomposition $\rho_{\varepsilon}(x-s) = \rho_0(x-s) + \varphi(x,s,\varepsilon)$ (for simplicity, in the following we drop the subscript $v$). If inserted in \eqref{eq: rho}, we obtain
\begin{equation}
    \begin{split}
        F(\varphi, \varepsilon,s) := &-\frac{1}{2}\partial_x^2\varphi + \frac{v^2}{2}\Bigl(-\varphi + \frac{1}{(\rho_0(x-s)+\varphi)^3}-\frac{1}{\rho_0^3(x-s)}\Bigr) - \bigl(1-\rho_0^2(x-s)\bigr)\varphi \\ & + 2\rho_0^2(x-s)\varphi +3 \rho_0(x-s)\varphi^2 + \varphi^3 + \varepsilon V(x)[\rho_0(x-s)+\varphi] =0.
    \end{split}
\end{equation}
Here, $F$ is defined as $F: U\times\mathbf{R} \times \mathbf{R} \to L^2(\mathbf{R})$, where $U :=\{u \in H^2(\mathbf{R}), \ ||u||_{H^2(\mathbf{R})}< \delta_0\} \subset H^2(\mathbf{R})$, for $\delta_0>0$ small enough. The map $F$ is $C^p$ in $U$ for any $p \geq 1$. In particular, its Frechét derivative in $\varphi=0$ and $\varepsilon=0$ is given by the operator $L: H^2(\mathbf{R}) \to L^2(\mathbf{R})$ defined as
\begin{equation}
    L := -\frac{1}{2}\partial_x^2-\frac{v^2}{2}\Bigl(1+\frac{3}{\rho_0^4(x-s)}\Bigr) - (1-\rho_0^2(x-s)) + 2\rho_0^2(x-s)
    \label{eq: operator_L}
\end{equation}
The equation $F(\varphi, \varepsilon,s)=0$ can thus be rewritten as
\begin{equation}
    F(\varphi,\varepsilon,s) = L\varphi + N(\varphi,s) + \varepsilon V(x) [\rho_0(x-s)+\varphi] =0,
\end{equation}
where $N(\varphi,s): U \times \mathbf{R} \to L^2(\mathbf{R})$ is defined as
\begin{equation}
    N(\varphi,s) := \frac{v^2}{2}\Bigl(\frac{6\rho_0^2(x-s)\varphi^2+8\rho_0(x-s)\varphi^3+3\varphi^4}{\rho_0^4(x-s)(\rho_0(x-s)+\varphi)^3}\Bigr)+\varphi^2(\rho_0(x-s)+\varphi) + 2\rho_0(x-s)\varphi^2,
\end{equation}
and we have $||N(\varphi,s)||_{L^2(\mathbf{R})} = o(||\varphi||_{H^2(\mathbf{R})})$ as $||\varphi||_{H^2(\mathbf{R})} \to 0$. In particular, we have $||N(\varphi,s)||_{L^2(\mathbf{R})} = O(||\varphi||^2_{H^2(\mathbf{R})})$ as $||\varphi||_{H^2(\mathbf{R})} \to 0$.\\
We can rewrite the operator in \eqref{eq: operator_L} as $L =-\frac{1}{2}\partial_x^2 + W(x)$, where $W(x) \to 2-2v^2$ exponentially fast as $|x| \to \infty$. By Kato-Rellich theorem, $L$ is a self-adjoint operator (see \cite[Proposition 4.8]{teta2018primer}), and by Weyl's lemma, and the exponential decay of $W(x)$, we have $\sigma_{ess}(L) = [2-2v^2,+\infty)$ (see \cite{pelinovsky_book} Appendix B.15 - Lemma B.5). We conclude that $L$ is a Fredholm operator of index zero (see \cite{edmunds_evans} Chapter IX, Theorem 1.6, and \cite{rabier_stuart}). Moreover, we have $L(\partial_x\rho_0(x-s)) =0$. Since zero must be a simple eigenvalue of $L$ (see \cite{pelinovsky_book} Corollary 4.1 of Chapter 4), we have $Ker(L) = Span\{\partial_x\rho_0(x-s)\}$. Using the Fredholm property, we can write $L^2(\mathbf{R}) = Ker(L) \bigoplus Ran(L)$, and we can apply the method of Lyapunov-Schmidt to the problem $F=0$, as in \cite[Chapter 2.7.6]{nirenberg}. We define $Q$ to be the projector from $L^2(\mathbf{R})$ onto $Ran(L)$. Then, $F(\varphi,\varepsilon,s) =0$ is equivalent to the equations $QF(\varphi, \varepsilon,s) =0$ and $(1-Q)F(\varphi,\varepsilon,s) =0$. By the implicit function theorem (\cite[Theorem 2.7.5]{nirenberg}) applied to $QF: U \subset H^2(\mathbf{R})\times \mathbf{R}\times \mathbf{R} \to Ran(L)$, there exists a unique smooth map $(\varepsilon,s) \to \varphi(\varepsilon,s) \in Ran(L)\cap H^2(\mathbf{R})$ such that $QF(\varphi(\varepsilon,s),\varepsilon,s) =0$ for $\varepsilon$ small enough. By Taylor's theorem \cite[Theorem 5.6.2]{cartan}, we have $||\varphi(\varepsilon,s)||_{H^2(\mathbf{R})} = O(\varepsilon)$ as $\varepsilon \to 0$. We now consider the equation $(1-Q)F(\varphi(\varepsilon,s), \varepsilon,s) =0$, usually referred to as bifurcation equation. It is equivalent to
\begin{equation}
    \Bigl(\partial_x\rho_0(x-s), N(\varphi(\varepsilon,s), s)\Bigr) + \varepsilon \Bigl(\partial_x \rho_0(x-s), V(x) (\rho_0(x-s) + \varphi(\varepsilon,s)) \Bigr) =0,
\end{equation}
where the parentheses correspond to the scalar product in $L^2(\mathbf{R})$. If divided by $\varepsilon$, the expression above can be written as
\begin{equation}
    G(\varepsilon,s) := \frac{1}{2}M'(s) + \tilde{G}(\varepsilon,s) =0,
\end{equation}
where $M'(s) = 2 \Bigl(\partial_x\rho_0(x-s), V(x)\rho_0(x-s)\Bigr)$ and $$\tilde{G}(\varepsilon,s) := \frac{1}{\varepsilon}\Bigl(\partial_x\rho_0(x-s), N(\varphi(\varepsilon,s),s)\Bigr) + \Bigl(\partial_x \rho_0(x-s), V(x) \varphi(\varepsilon,s) \Bigr).$$ Using the fact that $||N(\varphi(\varepsilon,s),s)||_{L^2(\mathbf{R})} = O(\varepsilon^2)$ as $\varepsilon \to 0$, we have that $G(\varepsilon,s)$ is smooth in $\varepsilon$ and $s$ for $\varepsilon$ small enough. Now, our goal is to apply the implicit function theorem to $G(\varepsilon,s) =0$. Consider $s_0 \in \mathbf{R}$ simple root of \eqref{eq: M'}. Then we have $G(0,s_0)=\frac{1}{2}M'(s_0)=0$, and $\partial_sG(0,s_0) = \frac{1}{2}M''(s_0) \neq 0$. Thus, there exists a smooth curve $s(\varepsilon) = s_0 + O(\varepsilon)$ such that $G(\varepsilon,s(\varepsilon)) =0$ for all $\varepsilon$ small enough. We conclude that the function $\varphi(\varepsilon,s(\varepsilon))$ satisfies $F(\varphi(\varepsilon,s(\varepsilon)), \varepsilon, s(\varepsilon)) =0$ for all $\varepsilon$ small enough. Thus, there exists a unique continuation of $\rho_0(x-s_0)$ to a solution $\rho_{\varepsilon}(x-s(\varepsilon)) = \rho_0(x-s(\varepsilon)) + \varphi(\varepsilon,s(\varepsilon))$ of \eqref{eq: rho} for $\varepsilon$ small enough. Moreover, $\rho_{\varepsilon}(x)$ and $\rho_0(x)$ are $\varepsilon$-close in the $L^{\infty}$-norm. 
\end{proof}
We can now prove Proposition \ref{prop: persistence_intro}, about the existence of time-independent solutions $\phi_{\varepsilon,v}$ to \eqref{eq: GP_V_main}.
\begin{proof}[Proof of Proposition \ref{prop: persistence_intro}]
    By Proposition \ref{prop: persistence_rho_0}, there exists a unique continuation of $\rho_{0,v}(x-s_0)$ in \eqref{eq: rho_0(x)} to a solution $\rho_{\varepsilon,v}(x-s_{\varepsilon})$ to \eqref{eq: rho}. We define $$\theta_{\varepsilon,v}(x-s_{\varepsilon}) := \frac{\pi}{2}+\int_{0}^{x}v\Bigl(1-\frac{1}{\rho^2_{\varepsilon}(t-s_{\varepsilon})}\Bigr)dt.$$ Then, the function $\phi_{\varepsilon,v}(x-s_{\varepsilon}):=\rho_{\varepsilon}(x-s_{\varepsilon})e^{i\theta_{\varepsilon}(x-s_{\varepsilon})}$ solves \eqref{eq: GP_V_main_stat}. Moreover, $|\phi_{\varepsilon,v}(x)|$ and $|\phi_{0,v}(x)|$ are $\varepsilon$-close in the $L^{\infty}$-norm by \eqref{eq: order_epsilon}, and we have
    \begin{equation}
        \lim_{\varepsilon \to 0} \bigl( ||\phi_{\varepsilon,v} - \phi_{0,v}||_{L^{\infty}(\mathbf{R})} + |s_{\varepsilon}-s_0| \bigr) =0.
    \end{equation}
    As $x \to \pm \infty$, the function $\theta_{\varepsilon}(x)$ converges to the limits
    \begin{equation*}
        \theta_{\varepsilon,\pm} := \frac{\pi}{2} + \int_{0}^{\pm \infty}v\Bigl(1-\frac{1}{\rho^2_{\varepsilon}(t-s_{\varepsilon})}\Bigr)dt.
    \end{equation*}
    Since, for any $x \in \mathbf{R}$ fixed, the function $\rho_{\varepsilon}(x-s_{\varepsilon})$ is smooth for $\varepsilon$ small enough, we can conclude the same for the map $\theta_{\varepsilon}(x-s_{\varepsilon})$. Similarly, the limits $\theta_{\varepsilon,\pm}$ are also smooth for $\varepsilon$ small enough. For any given $x \in \mathbf{R}$, the same regularity holds for the map $\varepsilon \to \phi_{\varepsilon,v}(x-s_{\varepsilon})$.
\end{proof}
We conclude this section with the following lemma, based on \cite[Lemma 2.9]{pelinovsky}, which reports a general property of gray modes.
\begin{lemma}
    Let $u_{\varepsilon,v}$ be a gray mode of Definition \ref{def: traveling_mode} for $v \in (-1,1) \backslash\{0\}$. Then for any $\varepsilon \neq 0$, it holds
    \begin{equation}
        \int_{\mathbf{R}}V'(x) [1-|u_{\varepsilon,v}(x)|^2]dx =0.
        \label{eq: no_drag}
    \end{equation}
    \label{lemma: no_drag}
\end{lemma}
\begin{proof}
    Following the discussion at the beginning of this section, we know that if $u_{\varepsilon,v}(x) = \rho_{\varepsilon,v}(x) e^{i\theta_{\varepsilon,v}}$ is a gray mode, than $\rho_{\varepsilon,v}$ solves \eqref{eq: rho} and $\theta_{\varepsilon,v}$ solves \eqref{eq: theta}. Consider the function
    \begin{equation*}
        E(\rho_{\varepsilon,v}, \partial_x\rho_{\varepsilon,v},x):= -\frac{1}{4}\bigl(\partial_x\rho_{\varepsilon,v}\bigr)^2 - \frac{v^2}{4}\bigl(\rho_{\varepsilon,v}^2+\frac{1}{\rho_{\varepsilon,v}^2}\bigr) + \frac{1}{4}\bigl(1-\rho_{\varepsilon,v}^2\bigr)^2+\frac{\varepsilon}{2}V(x)\bigl(\rho_{\varepsilon,v}^2-1\bigr).
    \end{equation*}
    Using \eqref{eq: rho}, we have
    \begin{equation*}
        \frac{d}{dx}E(\rho_{\varepsilon,v},\partial_x\rho_{\varepsilon,v},x) = \frac{\varepsilon}{2}V'(x) \bigl(\rho_{\varepsilon,v}^2-1\bigr).
    \end{equation*}
    Integrating the expression above on $x \in \mathbf{R}$, and using the asymptotic conditions of $\phi_{\varepsilon,v}$ from Definition \ref{def: traveling_mode}, we obtain \eqref{eq: no_drag}.
\end{proof}
\section{Linear Stability analysis}
\label{section: spectral_problem}
In this section we linearize equation \eqref{eq: GP_V_main} around a gray mode $\phi_{\varepsilon,v}(x-s_{\varepsilon})$ of Proposition \ref{prop: persistence_intro}, and we introduce the associated  matrix differential operator $\mathcal{L}_{\varepsilon}$. As we mentioned in the introduction, the eigenvalues of $\mathcal{L}_{\varepsilon}$ determine the stability of $\phi_{\varepsilon,v}(x-s_{\varepsilon})$. The goal of the next sections will be to show the existence of an eigenvalue $\lambda \in \mathbf{C}$ with positive real part, at least for $\varepsilon$ small enough.\\
From now on, the velocity $v$ will be a fixed number in $(-1,1)\backslash\{0\}$. For this reason, we will omit the subscript $v$ in denoting the gray modes. 
\subsection{Linearization of the equations}
We linearize \eqref{eq: GP_V_main} around $\phi_{\varepsilon}(x-s_{\varepsilon})$. We write $$u(x,t) = \phi_{\varepsilon}(x-s_{\varepsilon}) + [u(x)+iw(x)]e^{\lambda t} + [\overline{u}(x) + i \overline{w}(x)] e^{\overline{\lambda}t},$$
for $\lambda \in \mathbf{C}$ and $u(x), w(x) \in H^2(\mathbf{R})$. We insert the decomposition above in \eqref{eq: GP_V_main} and we retain only the linear terms in $u(x)$ and $w(x)$. For convenience, in the following we write $\phi_{\varepsilon}$ instead of $\phi_{\varepsilon}(x-s_{\varepsilon})$, whenever this notation does not lead to ambiguity. We obtain the following generalized eigenvalue problem
\begin{equation}
    \begin{pmatrix}
L_1 & D_1 \\
D_2 & L_2 
\end{pmatrix} \begin{pmatrix}
u \\
w 
\end{pmatrix} = \lambda J \begin{pmatrix}
    u\\
    w
\end{pmatrix} 
\label{spectral_problem}
\end{equation}
Where 
\begin{equation}
    \begin{split}
        D_1 & = +v\partial_x - 2 \Re\phi_\varepsilon \Im\phi_\varepsilon;   \\ 
        D_2  & = -v\partial_x - 2 \Re\phi_\varepsilon \Im\phi_\varepsilon; \\
        L_1 &= \frac{1}{2}\partial_{xx} +(1-|\phi_{\varepsilon}|^2) - 2(\Re\phi_{\varepsilon})^2 - \varepsilon V(x);\\
        L_2 &= \frac{1}{2}\partial_{xx} +(1-|\phi_{\varepsilon}|^2) - 2(\Im\phi_{\varepsilon})^2 - \varepsilon V(x).
\end{split} 
\label{spectral_operators}
\end{equation}
and 
\begin{equation}
J=
    \begin{pmatrix}
        0 & -1 \\
        1 & 0
    \end{pmatrix}.
\end{equation}
We can rewrite the eigenvalue problem (\ref{spectral_problem}) as
\begin{equation}
    HZ(x)=\lambda J Z(x).
    \label{def_H}
\end{equation}
for $Z(x) = [u(x), w(x)]^T$. Of great importance in the following will be the operator $\mathcal{L}_{\varepsilon}:=J^{-1}H$, which admits the following useful decomposition
\begin{equation}
    \mathcal{L}_{\varepsilon}Z = B\partial_x^2 Z +P \partial_xZ+N(x)Z
    \label{matrix_operator_L}
\end{equation}
where
\begin{equation}
    B= \begin{pmatrix}
        0 & 1/2 \\
        -1/2 & 0 
    \end{pmatrix} \ \ \ \ \ P = \begin{pmatrix}
        -v & 0 \\
        0 & -v
    \end{pmatrix}
    \label{eq: B and P}
\end{equation} while
\begin{equation}
    N(x) = \begin{pmatrix}
        -2\Re\phi_{\varepsilon}\Im\phi_{\varepsilon} & (1-|\phi_{\varepsilon}|^2)-2(\Im\phi_{\varepsilon})^2 - \varepsilon V(x) \\
        -(1-|\phi_{\varepsilon}|^2)+2(\Re\phi_{\varepsilon})^2 +\varepsilon V(x) & +2\Re\phi_{\varepsilon}\Im\phi_{\varepsilon}
    \end{pmatrix}.
    \label{eq: N(x)}
\end{equation}
We now give the following notion of spectral stability.
\begin{definition}
    We say that the solution $\phi_{\varepsilon}$ to \eqref{eq: GP_V_main} is spectrally unstable if there exists an eigenvector $[u,w]^T \in H^2(\mathbf{R}) \times H^2(\mathbf{R})$ of the eigenvalue problem \eqref{def_H} associated to an eigenvalue $\lambda$ with $\Re \lambda >0$. Otherwise, $\phi_{\varepsilon}$ is said to be spectrally stable.
\end{definition}
The next section is dedicated to the study of the eigenvalue problem \eqref{spectral_problem} for $\varepsilon=0$. We study some properties of the operator $\mathcal{L}_0$, and we define the Evans function associated to this operator. This requires writing the eigenvalue problem \eqref{spectral_problem} as a first order ODE system.
\section{The unperturbed problem}
\label{section: unperturbed_problem}
\subsubsection{Kernel and generalized kernel}
We begin by studying the kernel and generalized kernel of the operator $\mathcal{L}_{0}$.  
\begin{proposition}
We have the following properties:
\begin{enumerate}
    \item[(a)] $\mathcal{L}_0[-\Im\phi_0, \ \Re\phi_0]^T =0$;\\
    \item[(b)] $\mathcal{L}_0[\partial_x\Re\phi_0,\ 0]^T =0$;\\
    \item[(c)] $\mathcal{L}_0[\partial_v \Re\phi_0, \  \partial_v\Im\phi_0]^T = [\partial_x\Re\phi_0, \ 0]^T$;\\
    \item[(d)] $\mathcal{L}_0[-x\Im\phi_0, \ x\Re\phi_0]^T = [\partial_x\Re\phi_0,\ 0]^T-v[-\Im\phi_0, \Re\phi_0]^T$;\\
    \item[(e)] $\mathcal{L}_0[0, \ 1]^T = [\partial_x\Re\phi_0, \ 0]^T + 2v [-\Im \phi_0, \ \Re\phi_0]^T$;\\
    \item[(f)] $\mathcal{L}_0[-3\partial_v\Re\phi_0 -2x\Im\phi_0, \ -3\partial_v\Im\phi_0 + 2x\Re\phi_0 +1] =0$.
\end{enumerate}
\label{proposition_L}
\end{proposition} 
\begin{proof}
    Properties (a) and (b) are consequences of gauge and translational invariance, respectively. Properties (c), (d) and (e) follow by direct computation (in particular (c) follows by differentiation of \eqref{eq: GP_V_main_stat} at $\varepsilon=0$ w.r.t. the parameter $v$). Finally, (f) follows by combining the contributions from (c), (d) and (e). Notice that in (e) we use the fact that $\Im\phi_0(x) = v$.
\end{proof}
Next, we formulate the eigenvalue problem in \eqref{spectral_problem} at $\varepsilon=0$ as a first order ODE system, and we study some of its asymptotic properties as $|x| \to \infty$. We then define the Evans function associated to $\mathcal{L}_0$.
\subsubsection{The first order system at $\varepsilon=0$}
\label{section: eigenvalues}
If we write \eqref{spectral_problem} as a first order system, we obtain (hereafter $'=d/dx$)
\begin{equation}
    Y'(\lambda,x) = M(\lambda,x)Y(\lambda,x)
    \label{first_order_system}
\end{equation}
where $Y(\lambda,x)= [u,w,u',w']^T$ and
\begin{equation} M(\lambda,x) = 
    \begin{pmatrix}
    0 & 0 & 1 & 0 \\
    0 & 0 & 0 & 1 \\
    -2(1-|\phi_0|^2)+4(\Re\phi_0)^2 & -2\lambda+4\Re\phi_0\Im\phi_0 & 0 & -2v \\
    +2\lambda+4\Re\phi\Im\phi_0 & -2(1-|\phi_0|^2)+4(\Im\phi_0)^2 & +2v & 0 
\end{pmatrix}
\label{matrix_M}
\end{equation}
An important role is played by the asymptotic matrices
\begin{equation} M_{\pm}(\lambda) := 
    \begin{pmatrix}
    0 & 0 & 1 & 0 \\
    0 & 0 & 0 & 1 \\
    4(1-v^2) & -2\lambda\pm4v\sqrt{1-v^2} & 0 & -2v \\
    +2\lambda\pm4v\sqrt{1-v^2} & 4v^2 & +2v & 0 
\end{pmatrix},
\label{matrices_pm}
\end{equation}
which are defined as $M_{\pm}(\lambda) := \lim_{x \to \pm \infty}M(\lambda,x)$, and where the limit values are reached exponentially fast.\\ Our first goal is to gain some information about the eigenvalues $\gamma_1(\lambda), ..., \gamma_4(\lambda)$, for $i = 1,...,4$ of the matrices $M_{\pm}(\lambda)$. In particular, we are interested in their location in the complex plane and in their regularity as functions of $\lambda$.\\
We begin by computing the characteristic polynomial
\begin{equation}
    \det(M_{\pm}(\lambda)-\gamma\  \text{Id}) = \gamma^4-4(1-v^2)\gamma^2+8v\lambda\gamma+4\lambda^2.
     \label{ch_pol} 
\end{equation}
Notice that the matrices $M_+(\lambda)$ and $M_-(\lambda)$ share the same characteristic polynomial. If we take $\lambda =0$, the eigenvalues of $M_{\pm}(0)$ are the solutions of $\gamma^2-4(1-v^2)\gamma^2 =0$, namely
\begin{equation}
    \gamma =0 \quad \text{(double root)} \quad \text{and} \quad \gamma = \pm \sqrt{4(1-v^2)}.
    \label{eq: roots_at_zero}
\end{equation} In particular, the matrices $M_{\pm}(0)$ admit a double eigenvalue $\gamma=0$. On the other hand, as we will see in the next lemma, the matrices $M_{\pm}(\lambda)$ admit four distinct eigenvalues for $\lambda$ close but different from zero. In the terminology of Kato \cite[Chapter 2, Section 1]{kato}, $\lambda=0$ is an \textit{exceptional point}. It corresponds to a value of the parameter $\lambda$ for which the number of distinct eigenvalues of the matrices $M_{\pm}(\lambda)$ has changed (from four to three, in the present case). Since the matrices $M_{\pm}(\lambda)$ are linear in $\lambda$, the number of exceptional points is finite in any compact subset of the complex plane $\mathbf{C}$ \cite[Chapter 2, Section 1]{kato}. In particular, the exceptional point $\lambda=0$ is isolated. \\
Knowing the location of the exceptional points is useful in order to study the regularity of the eigenvalues of $M_{\pm}(\lambda)$ as $\lambda$ varies in the complex plane. In particular, if $\lambda$ is restricted to a simply-connected subdomain $D_s \subset\mathbf{C}$ containing no exceptional point, then the eigenvalues of $M_{\pm}(\lambda)$ can be written as
\begin{equation*}
    \gamma_1(\lambda), \gamma_2(\lambda), \gamma_3(\lambda), \gamma_4(\lambda),
\end{equation*}
all four functions being holomorphic in $D_s$, with $\gamma_i(\lambda) \neq \gamma_j(\lambda)$ for $i \neq j$ \cite[Chapter 2, Section 1]{kato}. However, some of these functions may fail to be holomorphic at the exceptional point $\lambda =0$, and they may constitute a branch of an analytic function with branch point at $\lambda=0$, as we discuss in Section \ref{section: eigenvalues_0}. In the next lemma we begin to study the properties of the eigenvalues of $M_{\pm}(\lambda)$ away from the exceptional point $\lambda =0$. 
\begin{lemma}
    There exists $r >0$ small enough such that, for any $\lambda$ in the half-disk 
    \begin{equation}
        D_r = \{\lambda \in \mathbf{C}, \ |\lambda| < r \quad \text{and} \quad \Re \lambda >0\},
    \end{equation}
    there are four distinct roots $\gamma_i(\lambda)$, for $i=1,...,4$, of \eqref{ch_pol}, with two roots having positive real part, and two roots having negative real part. If $\lambda \in D_r$ is real, then the four roots are real. Finally, the maps $\lambda \to \gamma_i(\lambda)$ can be chosen to be holomorphic in $D_r$.
    \label{lemma: distinct_roots}
\end{lemma}
\begin{proof}
    We consider again the characteristic polynomial in \eqref{ch_pol} with $\lambda = s$, for $s >0$ small enough. Then, the discriminant associated to this quartic function reads
    \begin{equation}
        D(s,v) = 16[4(1-v^2)]^4\cdot 4s^2 + 4 [4(1-v^2)]^3 \cdot (8vs)^2 + o(s^2).
    \end{equation}
    For all $v \in (-1,1)$ given, we can choose $s$ small enough so that $D(s,v) >0$, a condition which implies the existence of four distinct roots of \eqref{ch_pol} \cite{rees}. Moreover, since the coefficient $q: = -4(1-v^2)$ in front of the quadratic term in \eqref{ch_pol} is negative, and, for $s>0$ small enough, we have $4s^2 < q^2/4$, we have that all roots are real \cite{rees}. We conclude that, for $\lambda = s>0$ small enough, there are four distinct real roots of \eqref{ch_pol}. Regarding their sign, we proceed as follows: the product of the four roots equals $4s^2$, for any $s>0$. Moreover, for $\lambda =0$ the roots of \eqref{ch_pol} are given by \eqref{eq: roots_at_zero}. By choosing $s>0$ possibly smaller, we have that one root of \eqref{ch_pol} is strictly positive, being close to $\sqrt{4(1-v^2)}$, and another is strictly negative, being close to $-\sqrt{4(1-v^2)}$. This can be ensured by the continuity of the roots with respect to $s$. We conclude that the remaining two roots of \eqref{ch_pol} must be non-zero and must have opposite signs. We have thus shown that there exists $r>0$ such that, if $0<\lambda<r$, the matrices $M_{\pm}(\lambda)$ have four distinct eigenvalues, two with positive real part and two with negative real part. We now want to study the behavior of the eigenvalues as we move $\lambda$ in $D_r$, for a possibly smaller value of $r$.\\ The existence of four distinct roots for $\lambda \in (0,r)$ implies that $\lambda =0$ is an exceptional point for the matrices $M_{\pm}(\lambda)$. Since the matrices are linear in $\lambda$, this exceptional point is isolated. We consider $r>0$ possibly smaller, so that in the ball 
    \begin{equation}
        B_r: = \{\lambda \in \mathbf{C}, \quad |\lambda| < r\}, 
        \label{eq: ball_exceptional}
    \end{equation} there is only one exceptional point of $M_{\pm}(\lambda)$, which is $\lambda =0$. In particular, we have that $D_r$ does not contain any exceptional point. For this choice of $r>0$, the eigenvalues $\gamma_i(\lambda)$, for $i = 1,...,4$, can be expressed as holomorphic (and hence continuous) functions of $\lambda \in D_r$, with $\gamma_i(\lambda) \neq \gamma_j(\lambda)$, for $i \neq j$ (see \cite[Chapter 2, Section 1]{kato}). We now want to use the continuity property to show that no eigenvalue can cross the imaginary axis, as long as $\lambda$ varies in $D_r$. By contradiction, assume this is false. In particular, assume that one eigenvalue, say $\gamma_2(\lambda)$, has a real part that changes from positive to negative as we vary $\lambda$ in $D_r$. By continuity, there exists $\lambda_* \in D_r$ such that $\Re \gamma_2(\lambda_*) =0$. By solving \eqref{ch_pol} in $\lambda$, we have that one of the following relations holds
    \begin{equation}
       \lambda_* = \frac{\gamma_2(\lambda_*)}{2}(-2v+\sqrt{4-\gamma_2(\lambda_*)^2}) \qquad \text{or} \qquad  \lambda_* = \frac{\gamma_2(\lambda_*)}{2}(-2v-\sqrt{4-\gamma_2(\lambda_*)^2}).
    \label{lambda_sign_1} 
    \end{equation}
    Then, if $\Re \gamma_2(\lambda_*) = 0$, we have $\lambda_* \in i\mathbf{R}$, which is a contradiction. We conclude that, as long as $\lambda$ varies in $D_r$, the real part of each of the eigenvalues $\gamma_i(\lambda)$ remains non-zero and never changes sign.
\end{proof}
We conclude with the following observation.
\begin{lemma}
    The matrices $M_{\pm}(\lambda)$ have a purely imaginary eigenvalue only if $\lambda \in i\mathbf{R}$.
    \label{lemma: essential_spectrum}
\end{lemma}
\begin{proof}
    Assume $\gamma \in i\mathbf{R}$ is an eigenvalue of $M_{\pm}(\lambda)$, for a certain $\lambda \in \mathbf{C}$. From \eqref{ch_pol}, one of the following relations hold
    \begin{equation}
       \lambda = \frac{\gamma}{2}(-2v+\sqrt{4-\gamma^2}) \qquad \text{or} \qquad  \lambda = \frac{\gamma}{2}(-2v-\sqrt{4-\gamma^2}).
    \label{lambda_sign} 
    \end{equation}
    Then we conclude $\lambda \in i\mathbf{R}$.
\end{proof}
\begin{remark} Lemma \ref{lemma: essential_spectrum} gives us information on the essential spectrum of the operator $\mathcal{L}_0$. Indeed, for every $\lambda$ in the ball $B_r$ defined in \eqref{eq: ball_exceptional}, the following implication holds: $\lambda \in \sigma_{ess}(\mathcal{L}_0)$ if and only if the matrices $M_{\pm}(\lambda)$ admit a purely imaginary eigenvalue (see \cite[Theorem 3.1.11]{kapitula_book}). We conclude that $\sigma_{ess}(\mathcal{L}_0) \cap B_r \subset i\mathbf{R}$.
    \label{remark: ess_spectrum}
\end{remark}
By Lemma \ref{lemma: distinct_roots}, we have that,  for $\lambda \in D_r$, the eigenvalues $\gamma_i(\lambda)$, for $i = 1,...,4$, of $M_{\pm}(\lambda)$ are distinct and can be expressed as four are holomorphic functions. In particular, two of them have strictly positive real part, and two of them have strictly negative real part. By continuity, as $\lambda \to 0$ with $\lambda \in  D_r$, two roots converge to 0, one root converges to $\sqrt{4(1-v^2)}$, and the last one converges to $-\sqrt{4(1-v^2)}$. In order to fix the notation, we denote the eigenvalues of $M_{\pm}(\lambda)$ for $\lambda \in D_r$ in this way: \\
\begin{itemize}
    \item $\gamma_1(\lambda)$ has positive real part and tends to $+\sqrt{4(1-v^2)}$ as $\lambda \to 0$ with $\lambda \in D_r$;\\
    \item $\gamma_2(\lambda)$ has positive real part and tends to $0$ as $\lambda \to 0$ with $\lambda \in D_r$;\\
    \item $\gamma_3(\lambda)$ has negative real part and tends to $0$ as $\lambda \to 0$ with $\lambda \in D_r$;\\
    \item $\gamma_4(\lambda)$ has negative real part and tends to $-\sqrt{4(1-v^2)}$ as $\lambda \to 0$ with $\lambda \in D_r$.\\
\end{itemize}
In the next subsection we investigate the behavior of the eigenvalues of $M_{\pm}(\lambda)$ in the neighborhood of the exceptional point $\lambda=0$.
\subsubsection{Behavior of the eigenvalues at $\lambda =0$}
\label{section: eigenvalues_0}
We follow again the  discussion in \cite[Chapter 2, Section 1]{kato}. Consider the matrices $M_{\pm}(\lambda)$ in \eqref{matrices_pm} and consider a small disk $D$ in $\mathbf{C}$ near $\lambda = 0$, but excluding $\lambda = 0$ (and any other exceptional point). For $\lambda \in D$, the eigenvalues of $M_{\pm}(\lambda)$ are distinct and can be expressed by four holomorphic functions, which we denote again $\gamma_i(\lambda)$, for $i = 1,...,4$. If we move $D$ continuously around $\lambda = 0$, these four functions can be continued analytically. Then, if we bring $D$ back to its initial position after one revolution around $\lambda = 0$, we observe the four eigenvalues to have performed a permutation among themselves. A sub-group of eigenvalues that permute among themselves is called \textit{cycle} and the number of eigenvalues in a cycle is called \textit{period}. If an eigenvalue belongs to a cycle of period one, then after one revolution it goes back to its original value. In this case, the eigenvalue can be expressed as a holomorphic function in a neighborhood of $\lambda=0$. On the other hand, if an eigenvalue belongs to a cycle of period $p \geq 2$, it does not go back to its original value after one revolution. The elements of a cycle of period $p \geq 2$ constitute a branch of an analytic function (defined near $\lambda = 0$) with a branch point at $\lambda = 0$. They can be expressed as a \textit{Puiseux series}, with no negative exponents \cite[Chapter 2, Section 1, Equation 1.7]{kato}. In particular, all the eigenvalues $\gamma_i(\lambda)$, $i=1,...,4,$ are continuous at the exceptional point $\lambda=0$.\\
Now we consider the disk $D$ to be sufficiently small and sufficiently close to $\lambda=0$. In particular we consider the case $D \subset D_r$, where $D_r$ is defined in Lemma \ref{lemma: distinct_roots}. By continuity, if $\lambda \in D$, the eigenvalues $\gamma_1(\lambda)$ and $\gamma_4(\lambda)$ of $M_{\pm}(\lambda)$ are close to $\sqrt{4(1-v^2)}$ and $-\sqrt{4(1-v^2)}$, respectively. Similarly,  $\gamma_2(\lambda)$ and $\gamma_3(\lambda)$ are close to zero. As we move the disk $D$ continuously around $\lambda =0$ the roots $\gamma_i(\lambda)$, for $i=1,...,4$, can be continued analytically. If we keep $D$ sufficiently close to $\lambda=0$ during its revolution, we have that the eigenvalues $\gamma_1(\lambda)$ and $\gamma_4(\lambda)$ remain close to $\sqrt{4(1-v^2)}$ and $-\sqrt{4(1-v^2)}$, respectively, while $\gamma_2(\lambda)$ and $\gamma_3(\lambda)$ remain close to $0$. Thus, we conclude that $\gamma_1(\lambda)$ and $\gamma_4(\lambda)$ cannot perform a permutation with other eigenvalues. In particular, they belong to a cycle of period one (i.e. they return to their original value after one revolution). Thus the eigenvalues $\gamma_1(\lambda)$ and $\gamma_4(\lambda)$ can be expressed as holomorphic function in a neighborhood of $\lambda =0$. The eigenvalues $\gamma_2(\lambda)$ and $\gamma_3(\lambda)$, instead, may undergo a permutation among themselves after one revolution, and may form a cycle of period two. If this happens, these two eigenvalues cannot be expressed as holomorphic functions at $\lambda =0$. The goal of this subsection is to show that, in fact, each of the eigenvalues $\gamma_2(\lambda)$ and $\gamma_3(\lambda)$ belongs to a cycle of period one.\\ 
For simplicity, we assume for the moment that $v \in (0,1)$. Consider $r>0$ and $D_r$ as in Lemma \ref{lemma: distinct_roots}, and let $\gamma_2(\lambda)$ and $\gamma_3(\lambda)$ be the eigenvalues of the matrices $M_{\pm}(\lambda)$ defined for $\lambda \in D_r$, with $\Re\gamma_2(\lambda)>0$ and $\Re\gamma_3(\lambda)<0$ for every $\lambda \in D_r$. Choose $s \in (0,r)$ and denote $\gamma_2^* = \gamma_2(s)$ and $\gamma_3^* = \gamma_3(s)$. By Lemma \ref{lemma: distinct_roots}, we have $\gamma_2^*>0$ and $\gamma_3^*<0$. Using continuity, we choose $s>0$ small enough so that $(\gamma_{2}^*)^2 < 4(1-v^2)$ and $(\gamma_{3}^*)^2 < 4(1-v^2)$. From the characteristic polynomial in \eqref{ch_pol}, we deduce the relations
\begin{equation}
    s = \frac{\gamma_2^*}{2}(-2v+\sqrt{4-(\gamma_2^*)^2}) \qquad \text{and} \qquad s = \frac{\gamma_3^*}{2}(-2v-\sqrt{4-(\gamma_3^*)^2}).
    \label{eq: rel_lambda_gamma}
\end{equation}
Next, consider the subset $E_r$ of $D_r$, defined as
$$E_r := \{\lambda \in D_r, \quad \Im \lambda >0\}.$$ With some abuse of notation, we keep denoting $\gamma_2(\lambda)$ and $\gamma_3(\lambda)$ the eigenvalues of $M_{\pm}(\lambda)$ restricted to $E_r$. These are still analytic in $\lambda$ for $\lambda \in E_r$, and satisfy $\gamma_2(\lambda) \to \gamma_2^*$ and $\gamma_3(\lambda) \to \gamma_3^*$ as $\lambda \to s$ with $\lambda \in E_r$. Finally, we extend the eigenvalues $\gamma_2(\lambda)$ and $\gamma_3(\lambda)$ in an analytic way from $E_r$ to the set $B_r\backslash\mathbf{R}_+$, where $B_r$ is the ball defined in \eqref{eq: ball_exceptional} and $\mathbf{R}_+ = [0,+\infty)$. Recall that, by definition, $B_r$ contains a unique exceptional point, which is $\lambda=0$. This extension can be performed by following the procedure outlined above and in \cite[Chapter 2, Section 1]{kato}.\\
For $\gamma_2(\lambda)$ analytic in $B_r\backslash\mathbf{R}_+$, we define the function 
\begin{equation}
    \tilde{\gamma}_2(\theta) := \gamma_2(s e^{i\theta}), \qquad \text{for} \quad \theta \in I = (0,2\pi),
\end{equation}
 which is continuous in the interval $I$. In the following we assume that $|\tilde{\gamma}_2(\theta)|^2 < 4(1-v^2)$ for all $\theta \in I$. This condition can be obtained by initially choosing a small enough value of $s>0$, thanks to the continuity of the eigenvalues of $M_{\pm}(\lambda)$ at the exceptional point $\lambda=0$. The function $\tilde{\gamma}_2(\theta)$ satisfies $\tilde{\gamma}_2(\theta) \to \gamma_2^*$ as $\theta \to 0^+$. Then, following the discussion above, there are two possibilities for the limit $\theta \to 2\pi^-$: the first is that $\tilde{\gamma}_2$ goes back to its original value as we perform one revolution, i.e. $\tilde{\gamma}_2(\theta) \to \gamma_2^*$ as $\theta \to 2\pi^-$. In this case we have that $\gamma_{2}(\lambda)$ belongs to a cycle of period one, and hence is holomorphic in a neighborhood of $\lambda = 0$. The second is that $\tilde{\gamma}_2(\theta) \to \gamma_3^*$ as $\theta \to 2\pi^-$. In this case the eigenvalues $\gamma_2(\lambda)$ and $\gamma_3(\lambda)$ belong to a cycle of period two.\\
 We begin by stating the following lemma.
\begin{lemma}
Consider the map $f:I =(0,2\pi) \to \mathbf{C}$ defined as
\begin{equation}
    f(\theta) = se^{i\theta} - \frac{\tilde{\gamma}_2(\theta)}{2}(-2v + \sqrt{4-\tilde{\gamma}_2(\theta)^2})
    \label{eq: def_f}
\end{equation}
Then $f(\theta) = 0$ for any $\theta \in I$. 
\label{lemma: f(theta)}
\end{lemma}
\begin{proof}
For any value of $\theta \in I$, the complex number $\tilde{\gamma}_2(\theta)$ is a root of the characteristic polynomial \eqref{ch_pol} for $\lambda = s e^{i\theta}$. By \eqref{ch_pol}, for any $\theta \in I$ it holds one of the following relations: $$ se^{i\theta} = \frac{\tilde{\gamma}_2(\theta)}{2}(-2v+\sqrt{4-\tilde{\gamma}_2(\theta)^2}) \qquad \text{or} \qquad  se^{i\theta} = \frac{\tilde{\gamma}_2(\theta)}{2}(-2v-\sqrt{4-\tilde{\gamma}_2(\theta)^2}).$$ This implies that the possible values that the function $f(\theta)$ in \eqref{eq: def_f} can attain are
\begin{equation}
    f(\theta) = 0 \ \ \text{or} \ \ f(\theta) = -\tilde{\gamma}_2(\theta)\sqrt{4-\tilde{\gamma}_2(\theta)^2}
\end{equation}
Since $\tilde{\gamma}_2(\theta)$ is continuous for $\theta \in I$, then also $f(\theta)$ is continuous for $\theta \in I$. Moreover, we know that $\lim_{\theta \to 0^+}f(\theta) = 0$, by \eqref{eq: rel_lambda_gamma}, and that $|\tilde{\gamma}_2(\theta)\sqrt{4-\tilde{\gamma}_2(\theta)^2}| \geq C >0$ for all $\theta \in I$. By continuity, it must be $f(\theta) = 0$ for all $\theta \in I$.  
\end{proof}
We can now conclude the following Lemma.
\begin{lemma}
   The eigenvalue functions $\gamma_{2}(\lambda)$ and $\gamma_{3}(\lambda)$ belong to a cycle of period one. Hence they can be expressed as holomorphic functions in a neighborhood of $\lambda =0$. 
   \label{lemma: regular_roots}
\end{lemma}
\begin{proof}
    Suppose by contradiction that $\gamma_2(\lambda)$ belongs to a cycle of period two. Then $\tilde{\gamma}_2(\theta) \to \gamma_3^* \neq \gamma_2^*$, as $\theta \to 2\pi^-$. By Lemma \ref{lemma: f(theta)}, this implies that 
    \begin{equation}
        0=\lim_{\theta \to 2\pi^-}f(\theta) = s - \frac{\gamma_3^*}{2}(-2v + \sqrt{4-(\gamma_3^*)^2}). 
    \end{equation}
    From the relations in \eqref{eq: rel_lambda_gamma}, we have
    \begin{equation}
        \frac{\gamma_3^*}{2}\bigl(-2v - \sqrt{4-(\gamma_3^*)^2}\bigr) = \frac{\gamma_3^*}{2}\bigl(-2v + \sqrt{4-(\gamma_3^*)^2}\bigr),
    \end{equation}
    which means 
    \begin{equation}
        \gamma_3^*\sqrt{4-(\gamma_3^*)^2}=0
    \end{equation}
    %Then notice that $\tilde{\gamma}_2(0) = \gamma_2(r)=\gamma_2^*>0$. So, if we choose $\delta$ small enough, the left hand side of the equation above is strictly positive. Similarly, $\tilde{\gamma}_3(0) = \gamma_3(r) = \gamma_3^* < 0$ so that the right hand side of the equation above is strictly negative.
    which is a contradiction. We conclude that $\tilde{\gamma}_2(\theta) \to \gamma_2^*$ as $\theta \to 2\pi^-$, and that $\gamma_2(\lambda)$ belongs to a cycle of period one. 
\end{proof}
Finally, a similar result holds if $v \in (-1,0)$. We conclude that the four eigenvalues $\gamma_i(\lambda)$, for $i = 1,...,4$, of $M_{\pm}(\lambda)$ defined in Lemma \ref{lemma: distinct_roots} can be extended to a neighborhood of $\lambda =0$ in an analytic way.\\
In the next section we are going to compute the eigenvectors of $M_{\pm}(\lambda)$ associated with the eigenvalues $\gamma_i(\lambda)$, for $i = 1,...,4$, and we are going to study their regularity at $\lambda =0$.
\subsubsection{Regularity of the eigenvectors}
\label{subsection: eigenvectors}
In the half-disk $D_r$, the matrices $M_{\pm}(\lambda)$ have four distinct eigenvalues $\gamma_i(\lambda)$, for $i = 1,...,4$, which are holomorphic functions. Moreover, these functions can be extended in an analytic way in a neighborhood of $\lambda =0$, point at which the roots $\gamma_2(0)$ and $\gamma_3(0)$ coincide. The goal of this section is to study the eigenvectors associated with these eigenvalues, and to show that they can be chosen as analytic vector-valued functions of $\lambda$ in a neighborhood of the exceptional point $\lambda=0$. In particular, we show that they are analytic in the half-disk $D_r$, for a possibly smaller value of $r>0$.\\
For each $i = 1,...,4$, we denote by $v_i^{+}(\lambda)$ the eigenvector associated with the eigenvalue $\gamma_i(\lambda)$ for the matrix $M_{+}(\lambda)$. Similarly, for each $i=1,...,4$ we define $v^-_i(\lambda)$ the eigenvector associated with the eigenvalue $\gamma_i(\lambda)$ for the matrix $M_-(\lambda)$.\\
We begin by noticing that the eigenvalues $\gamma_1(\lambda)$ and $\gamma_4(\lambda)$ are simple at $\lambda =0$. As a general result from analytic perturbation theory \cite[Theorem 2]{greenbaum}, the eigenvectors $v_1^{\pm}(\lambda)$ and $v_4^{\pm}(\lambda)$ can be chosen to be analytic in a neighborhood of $\lambda =0$. Next, we compute explicitly the eigenvectors $v_2^{\pm}(\lambda)$ and $v_3^{\pm}(\lambda)$ as functions of $\gamma_2(\lambda)$ and $\gamma_3(\lambda)$, respectively. For example, $v_2^-(\lambda)$ can be obtained solving the equation $\bigl(M_-(\lambda)-\gamma_2(\lambda) \text{id}\bigr) \  v_2^-(\lambda) =0$. We obtain
\begin{equation}
   v_2^-(\lambda) = [-v, \ \  \frac{2\lambda - 4v\sqrt{1-v^2}+2v\gamma_2}{\gamma_2^2-4v^2}(-v), \ \  -v\gamma_2, \ \ \frac{2\lambda - 4v\sqrt{1-v^2}+2v\gamma_2}{\gamma_2^2-4v^2}(-v)\gamma_2]^T
        \label{eq: v_2}
\end{equation}
The above normalization has been chosen for later convenience. Similarly, one obtains 
\begin{equation}
   v_2^+(\lambda) = [-v, \ \  \frac{2\lambda + 4v\sqrt{1-v^2}+2v\gamma_2}{\gamma_2^2-4v^2}(-v), \ \  -v\gamma_2, \ \ \frac{2\lambda + 4v\sqrt{1-v^2}+2v\gamma_2}{\gamma_2^2-4v^2}(-v)\gamma_2]^T
\end{equation}
Since $\gamma_2(\lambda)$ is analytic in a neighborhood of $\lambda = 0$, with $\gamma_2(0) =0$, we have have that $v_2^{\pm}(\lambda)$ are also analytic in a neighborhood of $\lambda = 0$. For the eigenvectors $v_3^{\pm}(\lambda)$ we have 
\begin{equation}
     v_3^-(\lambda) = [-v, \ \  \frac{2\lambda - 4v\sqrt{1-v^2}+2v\gamma_3}{\gamma_3^2-4v^2}(-v), \ \  -v\gamma_3, \ \ \frac{2\lambda - 4v\sqrt{1-v^2}+2v\gamma_3}{\gamma_2^3-4v^2}(-v)\gamma_3]^T,
\end{equation}
and
\begin{equation}
     v_3^+(\lambda) = [-v, \ \  \frac{2\lambda + 4v\sqrt{1-v^2}+2v\gamma_3}{\gamma_3^2-4v^2}(-v), \ \  -v\gamma_3, \ \ \frac{2\lambda + 4v\sqrt{1-v^2}+2v\gamma_3}{\gamma_3^2-4v^2}(-v)\gamma_3]^T,
     \label{eq: v_3}
\end{equation}
 We conclude that also $v^{\pm}_3(\lambda)$ are analytic in a neighborhood of $\lambda=0$. Finally, we notice that $v_{2}^{-}(0) = v_3^-(0)$ and $v_{2}^{+}(0) = v_3^+(0)$, i.e. the Jordan canonical form of the matrices $M_{\pm}(0)$ has a block $\begin{pmatrix}
     0 & 1\\
     0 & 0
 \end{pmatrix}$, as in \cite{kapitula_dark}. 
 \begin{remark}
     For the case $v=0$, one can adopt the normalization in introduced in \cite[Equation (3.25)]{kapitula_dark}, which is obtained by imposing $\pm\frac{1}{2}(\gamma_2-v)$ as the first component of $v_2^{\pm}(\lambda)$, and similarly for $v_3^{\pm}(\lambda)$. In the following, the case $v=0$ will be recovered in the limit $v \to 0$.
 \end{remark}
\subsection{Definition of the Evans function} 
\label{section: def_evans} The goal of this section is to introduce the Evans function associated to \eqref{first_order_system}. We first define it in the half-disk $D_r$, and we later extend it to a neighborhood of the origin, inside the essential spectrum.\\
Consider the half-disk $D_r$ defined in Lemma \ref{lemma: distinct_roots}. In this set, we have four distinct eigenvalues $\gamma_i(\lambda)$, for $i=1,...,4$, of the matrices $M_{\pm}(\lambda)$, which are holomorphic functions. Similarly, we have four linearly independent eigenvectors $v_i^+(\lambda)$ for the matrix $M_+(\lambda)$ and four linearly independent eigenvectors $v_i^-(\lambda)$ for the matrix $M_-(\lambda)$, all holomorphic in $D_r$. We have the following lemma.
\begin{lemma}
   For any $\lambda \in D_r$, there exist four solutions $Y_i(\lambda,x)$, for $i=1,...,4$, to equation \eqref{first_order_system} such that
    \begin{equation}
        Y_{j}(\lambda,x) \to v^-_{j}(\lambda) \ e^{\gamma_{j}(\lambda)x} \ \ \ \text{as} \ \ x \to -\infty, \qquad \text{for} \quad j=1,2  
        \label{tail_1}
    \end{equation}
    and
    \begin{equation}
        Y_{j}(\lambda,x) \to v^+_{j}(\lambda) \ e^{\gamma_{j}(\lambda)x} \ \ \ \text{as} \ \ x \to +\infty, \qquad \text{for} \quad j=3,4  
        \label{tail_2}
    \end{equation}
    \label{lemma: ind_sol}
\end{lemma}
\begin{proof}
    For any $\lambda \in D_r$ the four roots $\gamma_i(\lambda)$ are distinct, with $\Re\gamma_{1,2}(\lambda)>0$ and $\Re\gamma_{3,4}(\lambda)<0$. The existence of four solutions with exponential tails as in \eqref{tail_1} and \eqref{tail_2} follows by the Coddington–Levinson’s Theorem for ODEs \cite[Chapter 3, Theorem 8.1]{coddington}.  
\end{proof}
We can now define the Evans function associated to \eqref{first_order_system} at $\varepsilon=0$.
\begin{definition}
    The Wronskian determinant of the four linearly independent solutions in Lemma \ref{lemma: ind_sol} at any $x \in \mathbf
    R$ is called the Evans function $E(\lambda)$ of the spectral problem \eqref{first_order_system}. 
    \label{evans_f}
\end{definition}
\begin{remark}
    Because the Wronskian determinant of any four particular solutions of the ODE \eqref{first_order_system} is independent of $x$, the values of the Evans function $E(\lambda)$ are independent of $x$.
\end{remark}
Using the same notation as in \cite{kapitula_dark}, we denote as fast modes $Y_{f}^{-}(\lambda,x)$ and $Y_{f}^{+}(\lambda,x)$ the solutions $Y_1(\lambda,x)$ and $Y_4(\lambda,x)$, respectively, from Lemma \ref{lemma: ind_sol}. We recall that the first one decays at $-\infty$ with decay rate $\gamma_{1}(\lambda)$ and the second decays at $+\infty$ with decay rate $\gamma_{4}(\lambda)$. Similarly, we denote as slow modes $Y_{s}^{-}(\lambda,x)$ and $Y_{s}^{+}(\lambda,x)$ the solutions $Y_2(\lambda,x)$ and $Y_3(\lambda,x)$, respectively. In this notation, the Evans function can be written, for $\lambda \in D_r$, as
\begin{equation}
    E(\lambda) = (Y_s^- \wedge Y_f^- \wedge Y_s^+ \wedge Y_f^+ )(\lambda,0).
\end{equation}
The Evans function $E(\lambda)$ is analytic for $\lambda$ away from the essential spectrum $\sigma_{ess}(\mathcal{L}_0)$ of the operator $\mathcal{L}_0$ in \eqref{spectral_problem} (see, for example, \cite[Theorem 4.1]{sandstede_review}). In particular,  $E(\lambda)$ is analytic in the half-disk $D_r$ (see Remark \ref{remark: ess_spectrum}). Using the fact that the eigenvalues $\gamma_i(\lambda)$ and eigenvectors $v_i^{\pm}(\lambda)$, for $i=1,...,4$, of $M_{\pm}(\lambda)$ are analytic at $\lambda=0$, we can extend the Evans function to a neighborhood of $\lambda =0$ in an analytic way.
\begin{lemma}
    The Evans function $E(\lambda)$ is analytically continued in $\lambda$ near $\lambda =0$.
    \label{lemma: cont_evans}
\end{lemma}
\begin{proof}
    Following \cite[Lemma 4.7]{pelinovsky}, we notice that ODE system in \eqref{first_order_system} depends analytically in $\lambda \in \mathbf{C}$ and the boundary conditions \eqref{tail_1} and \eqref{tail_2} are also analytic in $\lambda$ for $\lambda$ close to zero. Thus, at any fixed $x \in \mathbf{R}$, the Evans function is analytic in $\lambda$  for $\lambda$ close to zero, as the Wronskian determinant of the four fundamental solutions of Lemma \ref{lemma: ind_sol}. 
\end{proof}
\begin{remark}
    In Appendix \ref{Section: appendix} we give a more detailed proof of Lemma \ref{lemma: cont_evans}, which is based on  \cite[Proposition 2.7 and Lemma 2.10]{kapitula&sandstede}.
\end{remark}
In the next section we study the Taylor expansion of the Evans function around $\lambda=0$. The first step is to study $E(\lambda)$ at the point $\lambda =0$.
\subsubsection{The Evans function at $\lambda =0$} We begin by considering the ODE system in \eqref{first_order_system} at $\lambda = 0$. It reads
\begin{equation}
    Y'(0,x) = M(0,x)Y(0,x) 
    \label{ode_0_lambda}
\end{equation}
with 
\begin{equation} M(0,x) = 
    \begin{pmatrix}
    0 & 0 & 1 & 0 \\
    0 & 0 & 0 & 1 \\
    -2(1-|\phi_0|^2)+4(\Re\phi_0)^2 & +4\Re\phi_0\Im\phi_0 & 0 & -2v \\
    +4\Re\phi_0\Im\phi_0 & -2(1-|\phi_0|^2)+4(\Im\phi_0)^2 & +2v & 0 
\end{pmatrix}.
\label{matrix_M0}
\end{equation}
We look for its four fundamental solutions. The first one is associated to the invariance under translations:
\begin{equation}
    u_1(x):= [\phi_0'(x),0,\phi_0''(x),0]^T
\end{equation}
The second one is known to exists from the theory of ODE \cite{coddington}, even if its explicit form is unknown:
\begin{equation}
    u_2(x) := [u_2^1,u_2^2,u_2^3,u_2^4]^T =\text{unknown}
\end{equation}
The third one is associated to the invariance under phase shifts:
\begin{equation}
    u_3(x) := [-v, \Re\phi_0, 0 , \phi_0']^T
\end{equation}
The last one is the linearly growing mode $(f)$ in Proposition \ref{proposition_L}:
\begin{equation}
    u_4(x) := \frac{1}{2} [-2vx-3\partial_v\Re\phi_0, \  2x\Re\phi_0 -2 , \  -2v-3(\partial_v\Re\phi_0)',  \ 2\Re\phi_0 + 2x \phi_0']^T
\end{equation}
Since $M(0,x)$ is traceless, the Wronskian determinant of the four fundamental solutions to the ODE system in \eqref{ode_0_lambda} is a constant finite number. We normalize $u_2(x)$ such that
\begin{equation}
    u_1(x)\wedge u_2(x)\wedge u_3(x)\wedge u_4(x) =1, \ \ \ \text{for any} \ \ x \in \mathbf{R}.
\end{equation}
We deduce that  $|u_2(x)| \to \infty$ exponentially fast as $x \to \pm \infty$. \\
We observe that there exists a unique fast and a unique slow mode at $\lambda = 0$ for \eqref{first_order_system}. We set $$Y_f^{+}(0,x) = Y_f^{-}(0,x) = u_1(x) \qquad \text{and} \qquad Y_s^{+}(0,x) = Y_s^{+}(0,x) = u_3(x).$$ Then, we have $Y^-_f(0,x)e^{-\gamma_1(0)x} \to v^-_1(0)$ and $Y^-_s(0,x) \to v_2^-(0)$ as $x \to -\infty$. Similarly, we have $Y^+_f(0,x)e^{-\gamma_4(0)x} \to v^+_4(0)$ and $Y^+_s(0,x) \to v_3^+(0)$ as $x \to +\infty$. The extended Evans function at $\lambda =0$ is defined as
\begin{equation}
    E(0) =  (Y_s^- \wedge Y_f^- \wedge Y_s^+ \wedge Y_f^+ )(0,0),
\end{equation}
and we have $E(0)=0$. \\
In the next sections we compute derivatives of the form $\partial_{\lambda}^k E (0)$, for $k \in \mathbf{N}$. As we will see, the first non-zero term in the expansion will be at order $k=3$, i.e. the zero of $E(\lambda)$ at $\lambda =0$ is of order three, and this fact is related to the presence of phase shift and translational symmetries. \\
In practice, in order to compute derivatives of the Evans function, we first need to compute derivatives of the fast and slow modes, such as $\partial^k_{\lambda}(Y_{f}^{\pm} -Y_{f}^{\pm})$ or $\partial^k_{\lambda}(Y_{s}^{\pm} -Y_{s}^{\pm})$. The next subsection is devoted to the fast modes.
\subsubsection{Derivatives of the fast modes}
\label{section: der_fast_modes}
In this subsection we compute the first and second derivatives of the fast modes, i.e. $\partial_{\lambda}(Y_f^{\pm}-Y_f^{\pm})(\lambda,x)$ and  $\partial_{\lambda}^2(Y_f^{\pm}-Y_f^{\pm})(\lambda,x)$, at $\lambda =0$ and $x =0$.\\
    We begin by differentiating the ODE system \eqref{first_order_system} at $\lambda=0$ to obtain
    \begin{equation}
        (\partial_\lambda Y_f^{\pm})'(0,x) = M(0,x)\partial_{\lambda}Y_f(0,x) + \partial_{\lambda}M(0,x)Y_f^{\pm}(0,x),
        \label{eq: der_system}
    \end{equation}
    where we recall that $Y_{f}^{\pm}(0,x) = u_1(x)$ and the non-homogeneous term satisfies $\partial_{\lambda}M(0,x)Y_f^{\pm}(0,x) = [0,0, 0, 2\phi_0']^T$. A particular solution to the non-homogeneous ODE \eqref{eq: der_system} is 
    \begin{equation}
        f(x)= [\partial_v\Re\phi_0,1, (\partial_v\Re\phi_0)', 0]^T,
    \end{equation}
    where $\partial_v$ represents the derivative with respect to the parameter $v$. To this particular solution we can add any homogeneous solution $u_i(x)$ for $i = 1,..., 4$. If we require $\partial_\lambda Y_f^{\pm}(0,x)$ to decay to zero as $x \to \pm \infty$, we obtain
\begin{equation}
    \partial_\lambda Y_f^{\pm}(0,x) = f(x) \mp \frac{1}{\sqrt{1-v^2}} \ u_3(x). 
\end{equation}
In particular, we have that 
\begin{equation}
    (\partial_\lambda Y_f^{-} - \partial_\lambda Y_f^{+})(0,0) =  \frac{2}{\sqrt{1-v^2}} \ u_3(0). 
    \label{der_fast}
\end{equation}
\begin{remark}
    Also the homogeneous solution $u_1(x)$ could be added to $\partial_{\lambda}Y_f^{\pm}(0,x)$, without altering the decay at infinity. But, as we will see, adding any such term will not affect the final computation of $\partial^k_{\lambda}E(\lambda)$, for $k \in \mathbf{N}$.
\end{remark}
In order to obtain  $\partial^2_{\lambda}(Y_f^{+}-Y_f^{-})(0,0)$, we differentiate \eqref{first_order_system} twice and obtain
\begin{equation}
    (\partial_\lambda^2Y_f^{\pm})'(0,x) = M(0,x)\partial_\lambda^2Y_f^{\pm}(0,x) + 2\partial_\lambda M(0,x) \partial_\lambda Y_f^{\pm}(0,x).
\end{equation}
This time the non-homogeneous term reads 
\begin{equation}
    h^{\pm}(x):=2\partial_\lambda M(0,x) \partial_\lambda Y_f^{\pm}(0,x) = [0, \ 0, \ -4 \pm \frac{4}{\sqrt{1-v^2}}\Re\phi_0, \ 4\partial_v\Re\phi_0 \pm \frac{4v}{\sqrt{1-v^2}}]^T. 
\end{equation}
Using the variation of parameters, we write $\partial_\lambda^2Y_f^{\pm}(0,x)  = \sum_{i=1}^4c^{\pm}_i(x)u_i(x)$.  For our purposes, it is sufficient to compute only the coefficients $c_2^{\pm}(0)$. As we will see, the other coefficients do not contribute to the final computation. By Cramer's rule, we have $(c_2^{\pm})'(x) =\det(u_1(x),h^{\pm}(x), u_3(x),u_4(x))$. We obtain, after some computations,
\begin{equation}
    (c_2^{\pm})'(x) = -(1-v^2)\phi'_0(-4 \pm \frac{4}{\sqrt{1-v^2}}\Re\phi_0).
\end{equation}
Integrating the expressions above, and requiring $c_2^{\pm}(\pm\infty) =0$, we obtain
\begin{equation}
    [c_2^-(0) - c_2^+(0)] = 4(1-v^2)\sqrt{1-v^2}
\end{equation}
We conclude that
\begin{equation}
    \partial_\lambda^2(Y_f^{-}-Y_f^{+})(0,0) = \sum_{j = 1, 3, 4}\tilde{c}_ju_j(0) +  4(1-v^2)\sqrt{1-v^2} \ u_2(0)
    \label{eq: der_2_fast}
\end{equation}
for some constants $\tilde{c}_j$, with $j=1,3,4$.\\
The next subsections are devoted to computing the derivatives of the slow modes. Due to the fact that the eigenvalues $\gamma_{2,3}(\lambda)$, which determine the decay rate of the slow modes, have zero real part for $\lambda \in i\mathbf{R}$, we have to follow a different approach than the one we used for the fast modes. In particular, we cannot require $\partial_{\lambda}Y_{s}^{\pm}(0,x) \to 0$ as $x \to \pm \infty$. The approach we follow is based on \cite[Section 3.2]{kapitula_dark}. In order to introduce it, we first define the adjoint problem to \eqref{first_order_system}.
\subsubsection{The adjoint problem at $\lambda=0$} 
The adjoint problem relative to \eqref{first_order_system} is defined as the ODE system  
    \begin{equation}
        Z'(\lambda,x) = -M^*(\lambda,x)Z(\lambda,x),
        \label{eq: adjoint_problem}
    \end{equation}
where $M^*(\lambda,x)$ is the transpose conjugate of $M(\lambda,x)$. \\
As the problem in \eqref{first_order_system} is related to the operator $\mathcal{L}_0$ in \eqref{matrix_operator_L} (it is the first order system associated to the eigenvalue problem for $\mathcal{L}_0$), the adjoint problem in \eqref{eq: adjoint_problem} is related to the adjoint operator $\mathcal{L}_0^*$ as follows. Define the adjoint operator $\mathcal{L}_0^*$ as
\begin{equation}
    \mathcal{L}_0^* := B^T\partial_x^2 - P^T\partial_x + N^T(x),
\end{equation}
where the matrices $B,P,N(x)$ are defined in \eqref{eq: B and P} and \eqref{eq: N(x)}. Then the following proposition holds. 
\begin{proposition}
    Let $Z_2$ be a solution of $(\mathcal{L}_0^*-\lambda^*)(B^{-T}Z_2)=0$, for $\lambda^* \in \mathbf{C}$ and $B^{-T} = (B^T)^{-1}$. Define $Z_1 := -Z_2'+P^TB^{-T}Z_2$. Then, $Z:=(Z_1,Z_2)$ is a solution of the adjoint equation $Z'(\lambda,x) = -M(\lambda,x)^*Z(\lambda,x)$.
    \label{proposition: adjoint_solutions}
\end{proposition}
\begin{proof}
    Let's denote $\tilde{Z}_2= B^{-T}Z_2$. Since $(\mathcal{L}_0^*-\lambda^*)\tilde{Z}_2=0$, we have
    \begin{equation}
        B^T\tilde{Z}''_2 = -(N^T-\lambda^*I)\tilde{Z}_2 + P^T\tilde{Z}'_2.
        \label{Z_adjoint}
    \end{equation}
    Define $Z_1 = -Z_2'+P^TB^{-T}Z_2$, so that 
    \begin{equation}
        Z_1' = -B^T\tilde{Z}''_2+P^T \tilde{Z}'_2.
    \end{equation}
    Using \eqref{Z_adjoint}, this means that 
    \begin{equation}
        Z_1' = (N^T-\lambda^*I)B^{-T}Z_2. 
    \end{equation}
    Altogether we have
    \begin{equation}
        \begin{cases}
             Z_1' &= (N^T-\lambda^*I)B^{-T}Z_2 \\
             Z_2' &= -Z_1 + P^TB^{-T}Z_2
         \end{cases}
    \end{equation}
    which means $Z = (Z_1,Z_2)$ solves $Z' = -M^*(\lambda,x)Z$, where
    \begin{equation}
    M^*(\lambda,x) =  \begin{pmatrix}
        0 & -(N^T(x)-\lambda^* I)B^{-T} \\
        I & -P^TB^{-T}
    \end{pmatrix}.
\end{equation}
\end{proof}
Here we are interested in determining the fundamental solutions to \eqref{eq: adjoint_problem} at $\lambda =0$. These will be denoted as $u_i^A(x)$, for $i=1,...,4$, and, up to reordering and normalization, they satisfy (see \cite[Proposition 2.5]{kapitula_melnikov})
\begin{equation}
    u_i^A(x) \cdot u_j(x) = \delta_{ij} \qquad \forall x \in \mathbf{R}.
    \label{eq: orthogonal_u_i}
\end{equation}
In particular, we are interested in computing $u_2^A(x)$ and $u_4^A(x)$. Thanks to Proposition \ref{proposition: adjoint_solutions}, we can construct solutions to the adjoint problem \eqref{eq: adjoint_problem} at $\lambda =0$, by solving $\mathcal{L}_0^*(B^{-T}Z_2)=0$. By explicit computation, one can verify the identity
\begin{equation}
    \mathcal{L}_0^*(B^{-T}Z_2) = 2HZ_2,
\end{equation}
where $H$ is the linearized Hamiltonian defined in \eqref{def_H}. By means of the gauge invariance, we have that $Z_2= [-v,\Re \phi_0]$ satisfies $HZ_2=0$. Thus, we set
\begin{equation}
    Z_1 = -Z_2'+ P^TB^{-T}Z_2 = \begin{pmatrix}
        0 \\ -\phi_0'
    \end{pmatrix}+ \begin{pmatrix}
        0 & -2v \\
        2v & 0
    \end{pmatrix} \begin{pmatrix}
        -v \\
        \Re\phi_0
    \end{pmatrix} = \begin{pmatrix}
        -2v \Re\phi_0 \\
        -2v^2-\phi_0'
    \end{pmatrix},
\end{equation}
and we define the adjoint solution $ u_4^A(x)$ as
\begin{equation}
    u_4^A(x) :=  \frac{1}{1-v^2} [-2v \Re\phi_0, \ -2v^2-\phi_0', \ -v, \ \Re\phi_0 ]^T.
    \label{u_4_adj}
\end{equation}
On the other hand, using the invariance under translations,  if we take $Z_2 =  [\phi_0', 0]^T$, and we set $Z_1 = [ -\phi_0'',\ 2v\phi_0']^T$, we can define the adjoint solution $ u_4^2(x)$ to be
\begin{equation}
    u_2^A(x) := (1-v^2)[\phi_0'', -\ 2v\phi_0', \ -\phi_0', \ 0 ]^T.
    \label{u_2_adj}
\end{equation}
One can explicitly verify that $u_4^A(x)\cdot u_4(x) =1$. For $u_2^A(x)$, instead, one notices the identity
\begin{equation}
    u_2^A(0)\cdot u_2(0) = u_1(0) \wedge u_2(0) \wedge u_3(0) \wedge u_4(0) =1. 
\end{equation}

\begin{remark}
Relation \eqref{eq: orthogonal_u_i} and the fact that $u_2(x)$ diverges exponentially fast both at plus and minus infinity, implies that $u_2^A(x)$ decays to zero exponentially fast as $|x| \to \infty$.
\end{remark} 
\subsubsection{Derivative of slow modes} In this section we compute the derivatives $\partial_{\lambda}Y_s^{\pm}(\lambda,x)$ at $\lambda =0$ and $x=0$. Following \cite{kapitula_dark} and \cite{kapitula&sandstede}, we introduce the functions
\begin{equation}
    Z_s^+(\lambda,x) := Y_s^+(\lambda,x)e^{-\gamma_3(\lambda)x} \qquad \text{and} \qquad  Z_s^-(\lambda,x) := Y_s^-(\lambda,x)e^{-\gamma_2(\lambda)x}.
\end{equation}
We note that they satisfy $Z_s^+(\lambda,x) \to v_3^+(\lambda)$ as $x \to + \infty$ and $Z_s^-(\lambda,x) \to v_2^-(\lambda)$ as $x \to - \infty$. In particular, the functions $Z_s^{\pm}(\lambda,x)$ have limits at $\pm \infty$ which are differentiable functions of $\lambda$.\\
In the following we  abbreviate $Z_s^{\pm}(\lambda,x) = Y^{\pm}_s(\lambda,x)e^{-\gamma_{3,2}(\lambda)x}$. The equation satisfied by $Z^{\pm}_s(\lambda,x)$ is
\begin{equation}
    \partial_xZ_s^{\pm}(\lambda,x) = (M(\lambda,x)-\gamma_{3,2}(\lambda))Z_s^{\pm}(\lambda,x).
\end{equation}
We then introduce the following decomposition
\begin{equation}
    Z^{\pm}_s(\lambda,x) = v_{3,2}^{\pm}(\lambda) + Y^{\pm}_s(0,x) -v_{3,2}^{\pm}(0) + w^{\pm}(\lambda,x),
    \label{eq: ansatz_Z}
\end{equation}
where $w^{\pm}(\lambda,x)$ is assumed to decay exponentially fast as $x \to \pm \infty$ and to satisfy $w^{\pm}(0,x) = 0$. Our goal now is to compute $\partial_{\lambda}w^{\pm}(0,0)$, and to obtain $\partial_{\lambda}Y_s^{\pm}(0,0)$ from it.\\
By direct differentiation, one obtains
\begin{equation}
    \partial_{x}(\partial_{\lambda} w^{\pm}(\lambda,x))|_{\lambda = 0} = (M(0,x)- \gamma_{3,2}(0))\partial_{\lambda}Z_s^{\pm}(0,x) +\partial_{\lambda}M(0,x)  Z_s^{\pm}(0,x) - \frac{\partial\gamma_{3,2}}{\partial\lambda}(0)Z_s^{\pm}(0,x). 
\end{equation}
Using the identities $\partial_{\lambda}Z_s^{\pm}(0,x) = \partial_{\lambda}v_{3,2}^{\pm}(0) + \partial_{\lambda}w^{\pm}(0,x)$, with $\gamma_{2,3}(0)= 0$, and $Z_s^{\pm}(0,x) = Y_s^{\pm}(x)$, we obtain the inhomogeneous equation
\begin{equation}
    \partial_{x}(\partial_{\lambda} w^{\pm}(\lambda,x))|_{\lambda = 0} = M(0,x)\partial_{\lambda}w^{\pm}(0,x) + G_{\pm}(x),
    \label{eq: inhom_w}
\end{equation}
where
\begin{equation}
    G_{\pm}(x) = M(0,x) \partial_{\lambda}v_{3,2}^{\pm}(0) +\partial_{\lambda}M(0,x)  Y_s^{\pm}(0,x) - \frac{\partial\gamma_{3,2}}{\partial\lambda}(0)Y_s^{\pm}(0,x).
\end{equation}
The inhomogeneous term satisfies $G_{\pm}(x) \to 0$ exponentially fast as $x \to \pm \infty$, as can be seen by differentiating the relation $M_{\pm}(\lambda)v_{3,2}^{\pm}(\lambda) = \gamma_{3,2}v_{3,2}^{\pm}(\lambda)$ in $\lambda =0$. We try to solve \eqref{eq: inhom_w} by variation of parameters. Using the fundamental  solutions $u_i(x)$ of the homogeneous problem, and the adjoint solutions $u_i^A(x)$, for $i=1,...,4$, we obtain
\begin{equation}
    \partial_{\lambda}w^{\pm}(0,0) = c_1^{\pm}(0)u_1(0) + c_3^{\pm}(0)u_3(0) + u_2(0) \int_{\pm \infty}^{0}G_{\pm}(x)\cdot u_2^A(x) dx + u_4(0) \int_{\pm \infty}^{0}G_{\pm}(x)\cdot u_4^A(x) dx,
\end{equation}
where we required $\partial_{\lambda}w^{\pm}(0,x)$ to decay to zero as $x \to \pm \infty$. Thanks to the exponential decay of $G_{\pm}(x)$ and of the adjoint solutions $u_{2,4}^A$ at infinity, all the improper integrals above are well defined. Our goal now is to compute explicitly these two integrals (the coefficients $c_1^{\pm}(0)$ and $c_3^{\pm}(0)$ will not be needed). First, we notice that the product $Y_s^{\pm}(0,x)\cdot u_{2,4}^A(x) =0$, since $Y_s^{\pm}(0,x) = u_3(x)$. Thus, the last term in $G_{\pm}(x)$ does not contribute to the integrals. Next, by explicit computation, we obtain that the second term in $G_{\pm}$ contributes with
\begin{equation}
    \int_{\pm\infty}^0 \partial_\lambda M(0,x) Y_s^{\pm}(0,x) \cdot u_2^A(x) dx = -(1-v^2)^2 \qquad \text{and} \qquad  \int_{\pm\infty}^0 \partial_\lambda M(0,x) Y_s^{\pm}(0,x) \cdot u_2^A(x) dx = 0.
\end{equation}
To compute the contribution of the first term in $G_{\pm}(x)$ we use the identity
\begin{equation}
    (M(0,x)\partial_{\lambda}v_{3,2}^{\pm}(0)) \cdot u_i^A(x) =  - \partial_{\lambda}v_{3,2}^{\pm}(0) \cdot \frac{d}{dx} u_i^A(x),
\end{equation}
where we used the fact that $M(0,x)$ is real. We obtain
\begin{equation}
    \int_{\pm\infty}^0  M(0,x) \partial_{\lambda}v_{3,2}^{\pm}(0) \cdot u_2^A(x) dx = - \partial_{\lambda}v_{3,2}^{\pm}(0) \cdot u_2^A(0),
\end{equation}
since $u_2^A(\pm\infty) =0$, and
\begin{equation}
     \int_{\pm\infty}^0  M(0,x) \partial_{\lambda}v_{3,2}^{\pm}(0) \cdot u_4^A(x) dx = \partial_{\lambda}v_{3,2}^{\pm}(0) \cdot (u_4^A(\pm\infty) - u_4^A(0)  ). 
\end{equation}
We conclude that, for some constants $\tilde{c}^{\pm}_{1},\tilde{c}^{\pm}_3$, we have
\begin{equation}
    \begin{split}
        \partial_{\lambda}w^{\pm}(0,0) &= \tilde{c}^{\pm}_1 u_1(0) + \tilde{c}^{\pm}_3 u_3(0) + \bigl[-(1-v^2)^2 -\partial_{\lambda}v^{\pm}_{3,2}(0)\cdot u_2^A(0)\bigr] \ u_2(0)\\ & + [\partial_{\lambda}v^{\pm}_{3,2}(0) \cdot u_4^A(\pm\infty) \bigr] \ u_4(0) -\bigl[\partial_{\lambda}v^{\pm}_{3,2}(0)\cdot u_4^A(0)\bigr] \  u_4(0).
    \end{split}
    \label{eq: deriv_w}
\end{equation}
We can now prove the following lemma.
\begin{lemma}
    The slow solutions $Y_s^{\pm}(\lambda,x)$ to \eqref{first_order_system} satisfy
    \begin{equation}
        \partial_{\lambda}Y_s^{-} (0,0) = c_1^-   u_1(0) + c_3^-   u_3(0) - (1-v^2)^2 \  u_2(0) + \Bigl(\partial_{\lambda}v_{2}^{-}(0)\cdot u_4^A(-\infty) \Bigr) \ u_4(0)
    \end{equation}
    and 
     \begin{equation}
        \partial_{\lambda}Y_s^{+} (0,0) = c_1^+   u_1(0) + c_3^+   u_3(0) - (1-v^2)^2 \ u_2(0) + \Bigl(\partial_{\lambda}v_{3}^{+}(0) \cdot u_4^A(+\infty) \Bigr) \ u_4(0),
    \end{equation}
    for some constants $c_{1,2}^{\pm}$. In particular, it holds
    \begin{equation}
       \begin{split}
            \partial_{\lambda}(Y_s^{-} - Y_s^{+})(0,0)  & = c_1  u_1(0) + c_3  u_3(0) + [\partial_{\lambda}v^-_2(0) \cdot u_4^A(-\infty) - \partial_{\lambda}v^+_3(0) \cdot u_4^A(+\infty)] \ u_4(0) 
       \end{split}
       \label{eq: deriv_slow}
    \end{equation}
    for some new constants $c_{1,3}$. 
    \label{lemma: slow}
\end{lemma}
\begin{proof}
    From the decomposition \eqref{eq: ansatz_Z}, and the identity $Z^{\pm}_s(\lambda,0) = Y^{\pm}_s(\lambda,0)$, we have that
    \begin{equation}
        \partial_{\lambda}Y^-_s(0,0) = \partial_{\lambda}v_2^-(0) + \partial_\lambda w^-(0,0).
    \end{equation}
    and
     \begin{equation}
        \partial_{\lambda}Y^+_s(0,0) = \partial_{\lambda}v_3^+(0) + \partial_\lambda w^+(0,0).
    \end{equation}
    Then, we observe that
    \begin{equation}
        \partial_\lambda v_2^-(0) = \sum_{i = 1}^4[ \partial_\lambda v_2^-(0) \cdot u_i^A(0)] u_i(0) \qquad \text{and} \qquad \partial_\lambda v_3^+(0) = \sum_{i = 1}^4[ \partial_\lambda v_3^+(0) \cdot u_i^A(0)] u_i(0).
    \end{equation} 
    By using \eqref{eq: deriv_w}, we obtain the desired result.
\end{proof}
In order to compute the explicit expression of the coefficient in front of $u_4(0)$ in \eqref{eq: deriv_slow}, we need to compute $\partial_{\lambda} v^-_2(0)$ and $\partial_{\lambda} v^+_3(0)$. From the characteristic polynomial in \eqref{ch_pol}, we obtain the expansions
\begin{equation}
    \gamma_2(\lambda) = \frac{1}{1-v}\lambda + O(\lambda^2) \qquad \text{and} \qquad \gamma_3(\lambda) = -\frac{1}{1+v}\lambda + O(\lambda^2).
\end{equation}
Recalling the expression for the eigenvalues $v^-_2(\lambda)$ and $v^+_3(\lambda)$ in \eqref{eq: v_2} and \eqref{eq: v_3}, we obtain
\begin{equation}
    \partial_\lambda v_2^-(0) = [0, \ \frac{1}{2v(1-v)}, \ \frac{-v}{1-v}, \ - \frac{\sqrt{1-v^2}}{1-v}]^T,
\end{equation}
and 
\begin{equation}
        \partial_\lambda v_3^+(0) = [0, \ \frac{1}{2v(1+v)}, \ \frac{v}{1+v}, \ -\frac{\sqrt{1-v^2}}{1+v} ]^T.
    \end{equation}
    Deducing the explicit expression of $u_4^A(\pm\infty)$ from \eqref{u_4_adj}, we conclude
    \begin{equation}
        \partial_{\lambda}(Y_s^{-} - Y_s^{+})(0,0)   = c_1  u_1(0) + c_3 u_3(0) +  \frac{2}{(1-v^2)} \ u_4(0).
        \label{slow_modes}
    \end{equation}
\subsubsection{First order expansion of $E(\lambda)$} In this section we compute the first non-zero coefficient in the expansion of $E(\lambda)$ at $\lambda=0$. We already know that $E(0)=0$. Using the Leibniz rule with respect to the product "$\wedge$", the antisymmetric property of the Wronskian under exchange of columns, and the identity $Y_{f}^{+}(0,x) =Y_{f}^{-}(0,x)$ and the analogous one for $Y_s^{\pm}(0,x)$, the first derivative of the Evans function at $\lambda=0$ can be written as
\begin{equation}
    \partial_{\lambda}E(\lambda) = (\partial_{\lambda}(Y_s^--Y_s^+) \wedge Y_f^-\wedge Y_s^+\wedge Y_f^+)(0,0)+ (Y_s^- \wedge \partial_{\lambda}(Y_f^--Y_f^+) \wedge Y_s^+\wedge Y_f^+)(0,0).
\end{equation}
Since $Y_f^{+}(0,x) = Y_f^-(0,x)$ and $Y_s^+(0,x) = Y_s^-(0,x)$, we have $\partial_{\lambda}E(0)=0$. In computing the second derivative $\partial^2_{\lambda}E(0)$, one encounters terms of the form\footnote{One encounters also terms involving $\partial^2_{\lambda}Y_{s}^{\pm}$ or $\partial^2_{\lambda}Y_{f}^{\pm}$, but these equal zero.}
\begin{equation}
    (\partial_{\lambda}(Y_s^--Y_s^+)\wedge \partial_{\lambda}(Y_f^- - Y_f^+) \wedge Y_s^+ \wedge Y_f^+)(0,0).
\end{equation}
Using the identity $Y_s^{\pm}(0,x) = u_3(x)$, and the fact that $\partial_{\lambda}(Y_f^- - Y_f^+)(0,x)$ is proportional to $u_3(x)$ by \eqref{der_fast}, we obtain $\partial_{\lambda}^2E(0)=0$.\\
In order to compute $\partial_{\lambda}^3E(0)$, we can proceed as follows. We begin by considering the case in which only one derivative acts on the slow modes, and two derivatives act one the fast modes. The first possibility is that each fast mode has one derivative, i.e.
\begin{equation}
    (\partial_{\lambda}(Y_s^--Y_s^+) \wedge \partial_{\lambda}Y_f^-\wedge Y_s^+ \wedge \partial_{\lambda}Y_f^+)(0,0).
\end{equation}
Terms of this kind are always zero, since $\partial_{\lambda}Y_f^-(0,0)$ and $ \partial_{\lambda}Y_f^+(0,0)$ only differ by a multiple of $u_3(x)$, by \eqref{der_fast}. The second possibility is that two derivatives act on the same fast mode. Terms of this kind enter with multiplicity three, and their total contribution reads
\begin{equation}
     3(\partial_\lambda (Y_s^--Y_s^+) \wedge \partial_\lambda^2 (Y_f^--Y_f^+) \wedge Y_s^+ \wedge Y_f^+) (0,0).
\end{equation}
Using  \eqref{slow_modes} and \eqref{eq: der_2_fast}, the expression above equals $-24\sqrt{1-v^2}$.\\
We can consider now the case in which two derivatives act one the slow modes, and only one on the fast modes. If the two derivatives act on the same slow mode, we have a term of the form
\begin{equation}
    (\partial_{\lambda}^2(Y_s^--Y_s^+) \wedge \partial_{\lambda}(Y_f^--Y_f^+)\wedge Y_s^+ \wedge Y_f^+)(0,0).
\end{equation}
This term gives no contribution, since $\partial_{\lambda}(Y_f^--Y_f^+)(0,0)$ is proportional to $u_3(x)$, by \eqref{der_fast}. Finally, we consider the case in which each slow mode has one derivative. Terms of this kind enter with multiplicity six. Their total contribution reads
\begin{equation}
    6(\partial_{\lambda}Y_s^{+} \wedge 
    \partial_{\lambda}(Y_f^{-}-Y_f^+)\wedge \partial_{\lambda}Y_s^+\wedge Y_f^+)(0,0).
    \label{eq: comput_der}
\end{equation}
Using \eqref{slow_modes} we write $\partial_{\lambda}Y_s^-(0,0) = c_1 u_1(0) + c_3 u_3(0) + \frac{2}{(1-v^2)}u_4(0) + \partial_\lambda Y_s^+(0,0)$. Then, using \eqref{der_fast} for $\partial_{\lambda}(Y_f^{-}-Y_f^+)$, we can write \eqref{eq: comput_der} as
\begin{equation}
    6(\frac{2}{1-v^2} \ u_4(0) \wedge \frac{2}{\sqrt{1-v^2}} \ u_3(0) \wedge \partial_\lambda Y_s^+(0,0) \wedge u_1(0)).
\end{equation}
Using the expression of $\partial_{\lambda}Y_s^+(0,0)$ from Lemma \ref{lemma: slow}, we obtain that \eqref{eq: comput_der} equals $-24\sqrt{1-v^2}$.\\
Summing all the above contributions we have that $\partial_{\lambda}^3E(0) = -48\sqrt{1-v^2}$.\\ We conclude with the following proposition (recall the identity $P'_r(v) = 4\sqrt{1-v^2}$, where $P_r(v)$ represents the derivative of the renormalized momentum of $\phi_{0,v}$ in Remark \ref{remark: ren_mom}).
\begin{proposition}
    For $\lambda$ near zero the Evans function has the expansion
    \begin{equation}
        E(\lambda) = -2P'_r(v)\lambda^3 + O(\lambda^4). 
    \end{equation}
    \label{prop: expansion}
\end{proposition}
\begin{remark}
    In the limit $v \to 0$, we obtain $E(\lambda) = -8\lambda^2 + O(\lambda^4)$. If we compare this result with the one in \cite[Lemma 3.14]{kapitula_dark}, we observe a difference of a factor $2$. We believe this discrepancy arises from the fact that in \cite{kapitula_dark} the contribution coming from the term in \eqref{eq: comput_der} has been neglected. To support the result in Proposition \ref{prop: expansion}, we refer to \cite[Example 4.9]{pelinovsky}, where the explicit expression of $E(\lambda)$ at $v=0$ is reported, and where the expansion $E(\lambda) = -8\lambda^2 + O(\lambda^4)$ is verified (also numerically).
\end{remark}
\section{The perturbed case}
\label{section: perturbed_case} In this section we study equation \eqref{eq: GP_V_main} in the case $\varepsilon \neq 0$. As we showed in Proposition \ref{prop: persistence_intro}, for each $v \in (-1,1) \backslash \{0\}$, there exists a unique family of gray modes $\phi_{\varepsilon,v}(x-s_{\varepsilon})$, which arises as a regular perturbation of the gray soliton $\phi_{0,v}$, for $\varepsilon$ small enough. For each given $x \in \mathbf{R}$, we know that the map $ \varepsilon \to \phi_{\varepsilon,v}(x-s_{\varepsilon})$ is smooth in a neighborhood of $\varepsilon=0$. In addition, we know the existence of a real-valued smooth function $\varepsilon \to \omega(\varepsilon)$, defined for $\varepsilon$ small enough and satisfying $\omega(0) = v$, such that
\begin{equation}
    \lim_{|x|\to\infty}\Im\phi_{\varepsilon,v}(x) = \omega(\varepsilon)\qquad \text{and} \qquad \lim_{x \to \pm \infty}\Re \phi_{\varepsilon,v}(x) \to \pm \sqrt{1- \omega^2(\varepsilon)},
\end{equation}
where the limits are approached exponentially fast. The goal of this section is to use these properties of the gray modes $\phi_{\varepsilon,v}(x-s_{\varepsilon})$ in order to define and study the Evans function $E(\lambda, \varepsilon)$ associated with the operator $\mathcal{L}_{\varepsilon}$ in \eqref{matrix_operator_L}. \\
We begin by considering the spectral problem in \eqref{spectral_problem} with $\varepsilon\neq 0$, and by writing it as a first order system as
\begin{equation}
    Y'(\lambda,x) = M_{\varepsilon}(\lambda,x)Y(\lambda,x)
    \label{ode_epsilon}
\end{equation}
where
\begin{equation}
    M_\varepsilon(\lambda,x) :=\begin{pmatrix}
        0 & 0 & 1 & 0 \\
        0 & 0 & 0 & 1 \\
        -2(1-|\phi_\varepsilon|^2)+4\Re\phi_\varepsilon^2 + 2\varepsilon V(x) & -2\lambda+4\Re\phi_\varepsilon\Im\phi_\varepsilon & 0 & -2v \\
        +2\lambda+4\Re\phi_\varepsilon\Im\phi_\varepsilon\ & -2(1-|\phi_\varepsilon|^2)+4\Im\phi_\varepsilon^2 + 2\varepsilon V(x) & +2v & 0 \\
    \end{pmatrix}
\end{equation}
As before, we consider the asymptotic matrices $M_{\varepsilon,\pm}(\lambda) := \lim_{x\to\pm\infty} M_\varepsilon(\lambda,x)$, where
\begin{equation}
    M_{\varepsilon,\pm}(\lambda) = \begin{pmatrix}
        0 & 0 & 1 & 0 \\
        0 & 0 & 0 & 1\\
        4(1-\omega^2) & -2\lambda\pm4\omega\sqrt{1-\omega^2} & 0 & -2v\\
        +2\lambda\pm 4\omega\sqrt{1-\omega^2} & 4\omega^2 & +2v & 0\\
    \end{pmatrix}.
\end{equation}
The characteristic polynomial of $M_{\varepsilon,\pm}(\lambda)$ reads  $$\det(M_{\varepsilon,\pm}(\lambda)-\gamma I)= \gamma^4 - 4(1-v^2)\gamma^2 + 8v\lambda \gamma + 4\lambda^2,$$ for every $\omega \in \mathbf{R}$. In particular, the characteristic polynomial associated to $M_{\varepsilon,\pm}(\lambda)$ coincides with the one obtained in \eqref{ch_pol} for the case $\varepsilon=0$. By following the analysis in Section \ref{section: eigenvalues}, in particular by Lemma \ref{lemma: distinct_roots} and Lemma \ref{lemma: regular_roots}, we conclude that the matrices $M_{\varepsilon,\pm}(\lambda)$ admit, in the half-disk $D_r$,  four distinct eigenvalues $\gamma_i(\lambda)$, for $i = 1,...,4$ which are holomorphic functions of $\lambda$. Moreover, $\gamma_1(\lambda)$ and $\gamma_2(\lambda)$ have strictly positive real parts, and $\gamma_3(\lambda)$ and $\gamma_4(\lambda)$ have strictly negative real parts. Finally, the eigenvalues can be continued as holomorphic functions in the vicinity of $\lambda =0$.\\
We now study the behavior of the eigenvectors of $M_{\varepsilon,\pm}(\lambda)$ associated with the eigenvalues $\gamma_i(\lambda)$, for $i=1,...,4$. For simplicity, we keep the same notation as in Section \ref{subsection: eigenvectors}. For example, the eigenvector of $M_{\varepsilon,-}(\lambda)$ associated to $\gamma_1(\lambda)$ will be still denoted by $v_1^-(\lambda)$.\\
The eigenvectors $v_{1}^{\pm}(\lambda)$ and $v_4^{\pm}(\lambda)$, can be chosen to be analytic functions in the vicinity of $\lambda=0$, since the corresponding eigenvalues $\gamma_1(\lambda)$ and $\gamma_4(\lambda)$ do not split from a multiple eigenvalue, i.e. they do not belong to the $0$-group. Moreover, they can be chosen as smooth functions of $\varepsilon$, for $\varepsilon$ small enough.\\ The eigenvectors $v_2^{\pm}(\lambda)$ can be computed by solving the equation $ (M_{\varepsilon,\pm}(\lambda)-\gamma_2(\lambda))v_2^{\pm}(\lambda) = 0$. We obtain
\begin{equation}
     v_2^{\pm}(\lambda) = [-\omega, \ \  \frac{2\lambda \pm 4\omega\sqrt{1-\omega^2}+2v\gamma_2}{\gamma_2^2-4\omega^2} \ (-\omega), \ \  -\omega\gamma_2, \ \  \frac{2\lambda \pm 4\omega\sqrt{1-\omega^2}+2v\gamma_2}{\gamma_2^2-4\omega^2} \ (-\omega) \gamma_2]^T.
\end{equation}
Similarly, we obtain
\begin{equation}
     v_3^{\pm}(\lambda) = [-\omega, \ \  \frac{2\lambda \pm 4\omega\sqrt{1-\omega^2}+2v\gamma_3}{\gamma_3^2-4\omega^2} \ (-\omega), \ \  -\omega\gamma_3, \ \  \frac{2\lambda \pm 4\omega\sqrt{1-\omega^2}+2v\gamma_3}{\gamma_3^2-4\omega^2} \ (-\omega)\gamma_3]^T.
\end{equation}
We conclude that $v_2^{\pm}(\lambda)$ and $v_3^{\pm}(\lambda)$ are also analytic in the vicinity of $\lambda=0$ and smooth in $\varepsilon$, for $\varepsilon$ small enough.\\
We can now construct the Evans function associated to the problem \eqref{ode_epsilon}. We begin with the following Lemma, whose proof is similar to the one of Lemma \ref{lemma: ind_sol}. 
\begin{lemma}
   For any $\lambda \in D_r$ and any $\varepsilon \in (-\varepsilon_0,\varepsilon_0)$ with $\varepsilon_0>0$ small enough, there exist four solutions to \eqref{ode_epsilon}, which we keep denoting $Y_i(\lambda,x)$, with $i=1,...,4$,  such that
    \begin{equation}
        Y_{1,2}(\lambda,x) \to v^-_{1,2}(\lambda) \ e^{\gamma_{1,2}(\lambda)x} \ \ \ \text{as} \ \ x \to -\infty  
        \label{tail_1_bis}
    \end{equation}
    and
    \begin{equation}
        Y_{3,4}(\lambda,x) \to v^+_{3,4}(\lambda) \ e^{\gamma_{3,4}(\lambda)x} \ \ \ \text{as} \ \ x \to +\infty
        \label{tail_2_bis}
    \end{equation}
    \label{lemma: ind_sol_2}
    Here, $v_{1,2}^-(\lambda)$ and $v_{3,4}^{+}(\lambda)$ denote the eigenvectors of $M_{\varepsilon,-}(\lambda)$ and $M_{\varepsilon,+}(\lambda)$, associated with the eigenvalues $\gamma_{1,2}(\lambda)$ and $\gamma_{3,4}(\lambda)$, respectively.
    \label{lemma: ind_sol_epsilon}
\end{lemma}
We can now define the Evans function as follows.
\begin{definition}
    The Wronskian determinant of the four linearly independent solutions in Lemma \ref{lemma: ind_sol_epsilon} at any $x \in \mathbf
    R$ is called the Evans function $E(\lambda,\varepsilon)$ of the problem \eqref{ode_epsilon}. 
    \label{evans_f_epsilon}
\end{definition}
\begin{remark}
    As before, since the Wronskian determinant of any four particular solutions of the ODE \eqref{ode_epsilon} is independent of $x$, the values of the Evans function $E(\lambda, \varepsilon)$ are independent of $x$.
\end{remark}
As in Section \ref{section: def_evans}, we denote as fast modes $Y_{f}^{-}(\lambda,x)$ and $Y_{f}^{+}(\lambda,x)$ two of the solutions from Lemma \ref{lemma: ind_sol_epsilon}, the first decaying at $-\infty$ with decay rate $\gamma_{1}(\lambda)$ and the second decaying at $+\infty$ with decay rate $\gamma_{4}(\lambda)$. Similarly, we denote as slow modes $Y_{s}^{-}(\lambda,x)$ and $Y_{s}^{+}(\lambda,x)$ the other two solutions from Lemma \ref{lemma: ind_sol_epsilon}, the first decaying at $-\infty$ with decay rate $\gamma_{2}(\lambda)$ and the second decaying at $+\infty$ with decay rate $\gamma_{3}(\lambda)$. Then, the Evans function can be written, for $\lambda \in D_r$ and $\varepsilon \in (-\varepsilon_0,\varepsilon_0)$, as
\begin{equation}
    E(\lambda, \varepsilon) = (Y_s^- \wedge Y_f^- \wedge Y_s^+ \wedge Y_f^+ )(\lambda,0).
\end{equation}
As in Section \ref{section: def_evans}, by the analytic dependence of eigenvalues and eigenvectors in $\lambda$, and their smooth dependence in $\varepsilon$, we have the following lemma, whose proof is similar to one of Lemma \ref{lemma: cont_evans}.
\begin{lemma}
    The Evans function $E(\lambda, \varepsilon)$ can be continued analytically in $\lambda$ near $\lambda = 0$. It is also infinitely smooth in $\varepsilon$ near $\varepsilon = 0$. 
\end{lemma}
In the next subsections we compute a Taylor expansion of $E(\lambda, \varepsilon)$ at $\lambda=0$ and $\varepsilon=0$. In particular, we compute derivatives of the Evans function in $\lambda$ and $\varepsilon$. Regarding the derivatives in $\lambda$, we can use the results of the previous sections. In the next section we focus on the term $\partial_{\varepsilon}(Y_f^{-}-Y_f^{+})(0,0)$. 
\subsubsection{Derivatives of the fast modes with respect to $\varepsilon$} In this section we aim at computing $\partial_{\varepsilon}(Y_f^--Y_f^+)(0,0)$. We begin by differentiating in $\varepsilon$ the system in \eqref{ode_epsilon}, and by evaluating it at $\lambda =0$ and $\varepsilon=0$. We obtain, 
\begin{equation}
    (\partial_{\varepsilon}Y_f^{\pm})'(0,x) = M_{0}(0,x)\partial_{\varepsilon}Y_f^{\pm}(0,x) + \partial_{\varepsilon}M_0(0,x)Y_f^{\pm}(0,x).
    \label{eq: der_epsilon_ode}
\end{equation}
In order to evaluate $\partial_{\varepsilon}M_0(0,x)$, we use the properties of the gray mode $\phi_{\varepsilon}(x-s_{\varepsilon})$. We denote by $r(x)$ and $\psi(x)$ the first order terms in $\varepsilon$ of the expansion of the real and imaginary part of $\phi_{\varepsilon}(x-s_{\varepsilon})$, respectively. In other terms, we write
\begin{equation}
    \begin{split}
        \Re\phi_{\varepsilon}(x-s_{\varepsilon}) &= \Re\phi_0(x-s_0) + r(x) \varepsilon + O(\varepsilon^2) \\
        \Im\phi_{\varepsilon}(x-s_{\varepsilon}) &= \Im\phi_0(x-s_0) + \psi(x) \varepsilon + O(\varepsilon^2),
    \end{split} 
    \label{corrections}
\end{equation}
for all $x \in \mathbf{R}$. As we can verify by expanding \eqref{eq: GP_V_main_stat} at first order in $\varepsilon$, the functions $r(x)$ and $\psi(x)$ satisfy
\begin{equation}
    -L^0_1r - D_1^0 \psi = -V(x)\Re\phi_0(x-s_0)
    \label{eq: r},
\end{equation}
where $L^0_{1,2}$ and $D_{1,2}^0$ denote the operators in \eqref{spectral_problem}, evaluated at $\varepsilon = 0$.
We can now write $\partial_\varepsilon M_0(0,x)$ as 
\begin{equation}
    \partial_\varepsilon M_0(0,x) = \begin{pmatrix}
        0 & 0 & 0 & 0 \\
        0 & 0 & 0 & 0 \\
        12\Re(\phi_0)r + 4 \Im(\phi_0) \psi + 2V(x) & 4\Re(\phi_0) \psi + 4\Im(\phi_0) r  & 0 & 0 \\
        4\Re(\phi_0) \psi + 4\Im(\phi_0) r & 4\Re(\phi_0) r + 12 \Im(\phi_0) \psi + 2V(x) & 0 & 0 
    \end{pmatrix}.
    \label{matrix_derivative_epsilon}
\end{equation}
\begin{comment}
    In particular, recalling that $Y_f^{\pm}(0,x) = u_1(x)$, the non-homogeneous term in \eqref{eq: der_epsilon_ode} reads $\partial_{\varepsilon}M_0(0,x)Y_f^{\pm}(0,x) = [0,0,\   \bigl(12\Re\phi_0 r + 4 \Im\phi_0 \psi + 2V(x)\bigr)\phi_0', \ \ 4\phi_0'(\Re\phi_0 \psi + \Im\phi_0 r)].$
\end{comment}
(as before, we denote $\phi_{0}(x-s_0)$ simply by $\phi_0$). This time, the non-homogeneous term in \eqref{eq: der_epsilon_ode} reads
\begin{equation}
    \partial_\varepsilon M(0,x)u_1(x) = [0,\ 0, \ \bigl(12\Re(\phi_0)r + 4 \Im(\phi_0) \psi + 2V(x)\bigr) \phi_0', \bigl(4\Re(\phi_0) \psi + 4\Im(\phi_0) r\bigr) \phi_0' ]^T.
\end{equation}
We try to solve \eqref{eq: der_epsilon_ode} by variation of parameters, and we write
\begin{equation}
    \partial_\varepsilon Y_f^{\pm}(x) = \sum_{i=1}^4c_i^{\pm}(x) u_i(x).
\end{equation}
As before, we are only interested in the coefficients $c_2^{\pm}(0)$. Following the same strategy of Section \ref{section: der_fast_modes}, we have, by means of Cramer's rule, 
\begin{equation}
    (c_2^{\pm})'(x) = -(1-v^2)\phi'_0(12\Re\phi_0 r + 4 \Im\phi_0 \psi  + 2V(x))\phi'_0.
\end{equation}
Notice that the two derivatives coincide, and we write $ c_2'(x):=(c_2^{-})'(x) = (c_2^{+})'(x)$. Requiring $c_2^{\pm}(\pm\infty) =0$, we obtain
\begin{equation}
     [c_2^- -  c_2^+](0) = \int_{-\infty}^{+\infty}c_2'(t) \ dt = -(1-v^2)\int_{-\infty}^{+\infty}(\phi_0')^2 (12\Re\phi_0 r + 4 \Im\phi_0 \psi  + 2V(t))dt.  
     \label{c_2 epsilon}
\end{equation}
We now evaluate the expression above by using the relation in \eqref{eq: r}. We rewrite the potential term in \eqref{c_2 epsilon} as
\begin{equation}
        2\int_{-\infty}^{+\infty}(\phi'_0)^2 V(t)dt = 2\int_{-\infty}^{+\infty}\phi_0'(V(t)\Re\phi_0)'dt - 2\int_{-\infty}^{+\infty}\phi_0'V'(t)\Re\phi_0dt 
        \label{eq: comp_pot}
\end{equation}
By using \eqref{eq: r}, we have
\begin{equation}
    (V(t)\Re\phi_0)' = L_1^0r'+D_1^0\psi' -6 \Re(\phi_0)\phi_0'r -2 \Im(\phi_0) \phi_0'\psi.
\end{equation}
Using the fact that $L_1^0\phi_0'=0$ and $D_2^0\phi_0'=0$, the first term on the right-hand side of \eqref{eq: comp_pot} reads $$\int_{-\infty}^{+\infty}(\phi_0')^2(-12\Re\phi_0 r - 4 \Im\phi_0 \psi)dt.$$
Integrating by parts the second contribution in \eqref{eq: comp_pot}, we conclude that
\begin{equation}
    [c_2^- -  c_2^+](0) = (1-v^2)\int_{-\infty}^{+\infty} V''(t)(1-|\phi_0(t-s_0)|^2) dt = (1-v^2) M''(s_0),
\end{equation}
where $M''(s_0)$ has been introduced in Proposition \ref{prop: persistence_intro}. Finally, we obtain
\begin{equation}
    \partial_\varepsilon(Y_f^{-}-Y_f^{+})(0,0) = \sum_{j = 1, 3, 4}\tilde{c}_ju_j(0) +  (1-v^2) M''(s_0) \ u_2(0),
    \label{eq: der_fast_epsilon}
\end{equation}
for some constants $\tilde{c}_j$, for $j=1,3,4$.
\subsubsection{Derivative of the Evans function in $\lambda$ and $\varepsilon$} We can now compute the expansion of $E(\lambda,\varepsilon)$ at $\lambda=0$ and $\varepsilon=0$. The first derivative $\partial_{\varepsilon}E(0,0)$ is zero, thanks to the identity $Y_{f,s}^{-}(0,0)=Y_{f,s}^{+}(0,0)$. We consider the term $\partial_{\lambda}\partial_{\varepsilon}E(0,0)$. We obtain
\begin{equation}
    \partial_{\lambda}\partial_{\varepsilon}E(0,0) = \bigl(\partial_{\lambda}(Y_s^--Y_s^+)\wedge\partial_{\varepsilon}(Y_f^--Y_f^+)\wedge Y_s^+\wedge Y_f^+\bigr)(0,0).
\end{equation}
Notice that the term containing $\partial_{\lambda}(Y_f^--Y_f^+)(0,0)$ gives a vanishing contribution to the computation of $\partial_{\lambda}\partial_{\varepsilon}E(0,0)$, by \eqref{der_fast}. By means of \eqref{slow_modes} and \eqref{eq: der_fast_epsilon}, we obtain
\begin{equation}
    \partial_{\lambda}\partial_{\varepsilon}E(0,0) = -2M''(s_0).
\end{equation}
We can now state the following proposition.
\begin{proposition}
    For $\lambda$ and $\varepsilon$ near zero the Evans function has the expansion
    \begin{equation}
        E(\lambda,\varepsilon) = \lambda\bigl(-2P'_r(v)\lambda^2 -2M''(s_0) \varepsilon + O(\lambda^3, \lambda\varepsilon,\varepsilon^2)\bigr). 
    \end{equation}
    \label{prop: expansion_2}
\end{proposition}
We can conclude the following corollary.
\begin{corollary}
    Let $v \in (-1,1)\backslash\{0\}$. Let $\phi_{\varepsilon,v}(x-s_{\varepsilon})$ be the family of gray modes of Proposition \ref{prop: persistence_intro}, defined for $\varepsilon \in (-\varepsilon_0,\varepsilon_0)$, and consider the operator $\mathcal{L}_{\varepsilon}$ defined in \eqref{matrix_operator_L}. Then, $\mathcal{L}_{\varepsilon}$ admits a pair of eigenvalues that satisfy
    \begin{equation}
         \lambda^2 = -\frac{M''(s_0)}{P'_r(v)}\varepsilon + O(\varepsilon^{3/2}).
    \end{equation}
    Assume without loss of generality $\varepsilon>0$. Then the gray mode $\phi_{\varepsilon,v}(x-s_{\varepsilon})$ is spectrally unstable for $M''(s_0)<0$ and $\varepsilon$ small enough.
\end{corollary}
We conclude this section by considering the case of the Gaussian external potential $V_{1}$ defined in \eqref{eq: pot_laser}. As we mentioned in the introduction, this kind of potential is of physical relevance, as it models the repulsive effect of a laser beam moving on the BEC. As we will see below,  the gray modes which bifurcate from the gray solitons are spectrally unstable.\\
We begin by noticing that $V_{1}$ is an even function. Thus, we have that $s_0=0$ is a critical point of the effective potential $M(s)$, i.e. we have $M'(0) =0$. In Figure \ref{fig: M''} we plot, for $v = 0.5 $, the value of $M''(0)$ as a function of the variance $\sigma >0$ of $V_{1}$, for $\sigma \in (0,10)$. We observe that $M''(0)$ takes negative values for all $\sigma$. By a change of variables in $M''(0)$, one can show that the same conclusion holds for any $v \in (-1,1) \backslash\{0\}$. We conclude that, for any $v \in (-1,1) \backslash\{0\}$, there exists a unique family of gray modes $\phi_{\varepsilon,v}(x-s_{\varepsilon})$ that bifurcates from the gray soliton $\phi_{0,v}(x)$. Moreover, for $\varepsilon>0$ small enough, the elements of this family are spectrally unstable solutions to \eqref{eq: GP_V_main}. 
\begin{figure}
    \centering
    \includegraphics[width=0.8\linewidth]{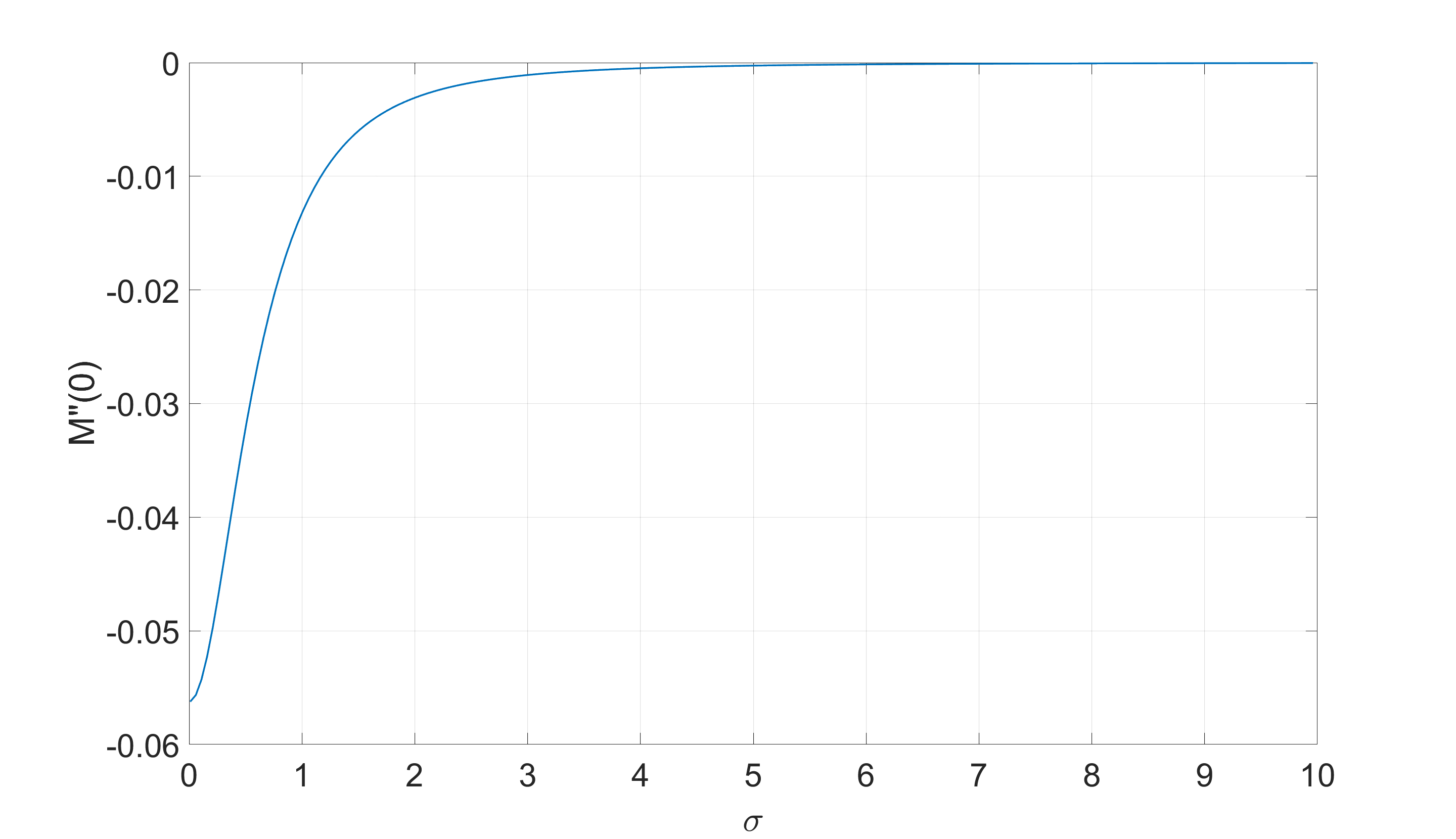}
    \caption{Plot of $M''(0)$ as a function of $\sigma$ for the case of the gaussian potential $V_1$, and for a velocity $v=0.5$.}
    \label{fig: M''}
\end{figure}
\section{The case of a repulsive short-range potential}
\label{section: delta_potential}
In this section we study equation \eqref{GP_delta_intro}. This equation is obtained from \eqref{eq: GP_V_main} by replacing the smooth potential $V$ with a repulsive delta potential, i.e. for $\varepsilon V(x) = g\delta(x)$, with $g >0$. In this case, thanks to the short-range nature of the potential, sharp existence conditions and explicit expressions for time-independent solutions are available. The goal of this section is to show the spectral instability of those time-independent solutions that bifurcate from the gray solitons $\phi_{0,v}$. \\
Due to the presence of the delta potential, time-independent solutions to \eqref{GP_delta_intro} are not smooth functions at $x=0$. At the origin they satisfy a jump discontinuity on the first derivative, which reads
\begin{equation}
    \partial_x\phi(0^+)-\partial_x\phi(0^-) = 2g\phi(0).
    \label{eq: jump}
\end{equation}
Using the fact that the delta potential vanishes away from the origin, we can construct time-independent solutions to \eqref{GP_delta_intro} by joining together at $x=0$ two copies of a displaced gray soliton (which solves \eqref{GP_delta_intro} with $g=0$), in such a way that continuity and \eqref{eq: jump} are satisfied. This approach has been followed by Hakim in \cite{hakim1997nonlinear} and permits to obtain an explicit expression for all time-independent solutions to \eqref{GP_delta_intro}, as we present below. \\
First of all, and similarly to the case of smooth potentials, no time-independent solution to \eqref{GP_delta_intro} exists for velocities $|v| \geq 1$ (see \cite{hakim1997nonlinear, maris2003}). For $v \in (-1,1) \backslash\{0\}$ and $g>0$ given, time-independent solutions to \eqref{GP_delta_intro} may be written as $\phi_{\xi,v}(x) = \rho_{\xi,v}(x)e^{i\theta_{\xi,v}(x)}$, where 
\begin{equation}
    \begin{split}
        &\rho_{\xi,v}(x) = \sqrt{v^2+(1-v^2)\tanh^2\bigl(\sqrt{1-v^2}(|x|+\xi)\bigr)}  \\ & \\&
    \theta_{\xi,v}(x) = \frac{\pi}{2} + v\int_0^x \Bigl(1-\frac{1}{\rho_{\xi,v}^2(s)}\Bigr)ds,
    \label{eq: modulus_delta}
    \end{split}
\end{equation}
(see \eqref{eq: rho} and \eqref{eq: theta}), and where $\xi >0$ is a displacement parameter that satisfies 
\begin{equation}
    g = (1-v^2)^{3/2}\frac{\tanh(\sqrt{1-v^2}\xi)}{v^2+\sinh^2(\sqrt{1-v^2}\xi)}. 
    \label{gamma_xi}
\end{equation}
Relation \eqref{gamma_xi} is obtained by imposing the jump condition \eqref{eq: jump} for $\phi_{\xi,v}$. By studying this relation one can deduce the following proposition.
\begin{proposition}[\cite{hakim1997nonlinear,maris2003,pham_3}]
    For any given velocity $v \in (-1,1)\backslash\{0\}$, there exists a critical strength $g_{cr}=g_{cr}(v)>0$ such that the following holds. If $g \in (0,g_{cr})$, there are two distinct solutions $\xi_1,\xi_2>0$ to \eqref{gamma_xi}. If $g = g_{cr}$ there is a unique solution to \eqref{gamma_xi}. If $g > g_{cr}$ there are no solutions to \eqref{gamma_xi}. Correspondingly, for $g \in (0,g_{cr})$ there are two distinct time-independent solutions to \eqref{GP_delta_intro}; for $g=g_{cr}$ there is a unique time-independent solution, while for $g >g_{cr}$ no time-independent solutions exist. Moreover, the value of $g_{cr}(v)$ is explicit and it holds
    \begin{equation}
       g_{cr}(v)= \frac{4(1-v^2)}{\sqrt{2}}\frac{[\sqrt{1+8v^2}-(1+2v^2)]^{1/2}}{4v^2-1+\sqrt{1+8v^2}}.
    \end{equation}
    \label{propo: existence}
    \end{proposition}
    \begin{remark}
 The bifurcation result in Proposition \ref{propo: existence} can be formulated from a different perspective. For any given value of $g>0$, there exists a critical velocity $v_{cr} = v_{cr}(g) \in (0,1)$ such that two distinct solutions exist for $|v| < v_{cr}$; a single solution exists for $|v|=v_{cr}$; no solutions exist for $|v|> v_{cr}$. This formulation highlights the fact that, in the presence of an external potential, the value of the critical velocity for the existence of time-independent solutions becomes strictly less than one, namely $v_{cr}(g) < v_{cr}(0)=1$, for any $g >0$.
\end{remark}
    In the following, it is convenient to use the displacement $\xi>0$ as the free parameter of the problem, playing the role of $\varepsilon$ in the previous sections, while keeping $v \in (-1,1) \backslash \{0\}$ as fixed. In particular, and similarly to \cite{pham-brachet, pham_3}, we see the strength $g$ as a function $g=g(\xi)$ by \eqref{gamma_xi}. The functions $\phi_{\xi,v}$, then,  will solve equation \eqref{GP_delta_intro} for parameters $v$ and $g=g(\xi)$.\\
As the displacement $\xi$ varies from zero, the solutions $\phi_{\xi,v}$ bifurcate from the gray soliton $\phi_{0,v}$. On the other hand, as $\xi \to +\infty$, the solution $\phi_{\xi,v}$ converges to the constant solution $u=1$. For each $v \in (-1,1)\backslash \{0\}$ fixed, an important quantity is the critical displacement $\xi_{cr}>0$ defined by the relation $g'(\xi_{cr}) =0$, and which admits the explicit expression
\begin{equation*}
    \xi_{cr} = \frac{\text{argcosh}(\frac{1+\sqrt{1+8v^2}}{2})}{2\sqrt{1-v^2}}.
\end{equation*}
The reason for its importance is the following: the numerical simulations of Hakim \cite{hakim1997nonlinear} and the numerical study by Pham and Brachet \cite{pham-brachet, pham_3} suggest that the solutions $\phi_{\xi,v}$ are unstable for $0<\xi < \xi_{cr}$ and stable for $\xi > \xi_{cr}$. At the critical value $\xi = \xi_{cr}$ the solution $\phi_{\xi_{cr},v}$ is believed to be spectrally stable, with neutral modes given by $i\phi_{\xi_{cr},v}(x)$ and  $\partial_{\xi}\phi_{\xi,v}(x)|_{\xi = \xi_{cr}}$. In our previous work \cite{antonelli_caliaro}, we showed the orbital stability of the states $\phi_{\xi,v}$ for $\xi > \xi_{cr}$. The goal of this section is to show the spectral instability of $\phi_{\xi,v}$ for $\xi \in (0,\xi_{cr})$ and small enough.\\
We begin by noticing that the states $\phi_{\xi,v}$ defined in \eqref{eq: modulus_delta} are regular perturbations of the dark solitons $\phi_{0,v}$. Their real and imaginary parts admit the following expression (as before, in the following we consider $v$ as fixed, and we denote $\phi_{\xi,v}$ simply by $\phi_{\xi}$):
\begin{equation}
     \Im \phi_{\xi}(x) = \frac{v^2}{\rho_{\xi}(0)} + \frac{(1-v^2)\tanh(\sqrt{1-v^2}(|x|+\xi))\tanh(\sqrt{1-v^2}\xi)}{\rho_{\xi}(0)} 
\end{equation}
and
\begin{equation}
      \Re \phi_{\xi}(x) = \text{sign}(x)\frac{v\sqrt{1-v^2}}{\rho_{\xi}(0)}(\tanh(\sqrt{1-v^2}(|x|+\xi))-\tanh(\sqrt{1-v^2}\xi)),
\end{equation}
where $\rho_{\xi}(0) = \sqrt{v^2+(1-v^2)\tanh^2(\sqrt{1-v^2}\xi)}$.
We thus observe\footnote{The map $g(\xi)$ and the family of solutions $\phi_{\xi}$, originally defined for $\xi>0$, can be extended to $\xi \leq 0 $ in a smooth way. In particular, for $\xi \leq 0$, the functions $\phi_{\xi}$ are still solutions to \eqref{GP_delta_intro} with $v \in (-1,1)\backslash\{0\}$ and $g(\xi)$.} that $\Re \phi_{\xi}(x)$ and $\Im\phi_{\xi}(x)$ are smooth functions of $\xi$ for any $x \in \mathbf{R}$. Moreover, the limits as $x \to \pm \infty$ read
\begin{equation}
    \begin{split}
        &\Im \phi_{\xi}(x) \to  \frac{v^2+(1-v^2)\tanh(\sqrt{1-v^2}\xi)}{\rho_{\xi}(0)} := \omega(\xi) \\ &    \Re \phi_{\xi}(x) \to \pm \frac{v\sqrt{1-v^2}}{\rho_{\xi}(0)}(1-\tanh(\sqrt{1-v^2}\xi)) = \pm \sqrt{1-\omega^2(\xi)}.
    \end{split}
\end{equation}
In particular, both limits are smooth functions at $\xi=0$, and they are approached exponentially fast.\\
The next step consists in linearizing \eqref{GP_delta_intro} around the state $\phi_{\xi}$. Similarly to Section \ref{section: spectral_problem}, we obtain an eigenvalue problem for the operator $\mathcal{L}_{\xi}$. The latter coincides with $\mathcal{L}_{\varepsilon}$ in \eqref{matrix_operator_L} upon replacing $\phi_{\varepsilon}$ with $\phi_{\xi}$ and $\varepsilon V(x)$ with $g(\xi)\delta(x)$. The spectral instability of $\phi_{\xi}$ is determined by the spectrum of $\mathcal{L}_{\xi}$: if $\mathcal{L}_{\xi}$ admits an eigenvalue with positive real part, then $\phi_{\xi}$ is a spectrally unstable solution. In order to determine the existence of an unstable eigenvalue we formally apply the results in Section \ref{section: perturbed_case}. Notice that this time the effective potential reads 
\begin{equation*}
    M_{\delta}(s):= \int_{\mathbf{R}}\delta(t)[1-|\phi_0(x-s)|^2]ds = 1-|\phi_0(s)|^2, \qquad s \in \mathbf{R}.
\end{equation*}
We conclude the following proposition.
\begin{proposition}
    Let $v \in (-1,1)\backslash\{0\}$. Let $\phi_{\xi,v}(x)$ be the family of time independent solutions to \eqref{GP_delta_intro} defined in \eqref{eq: modulus_delta}. Consider the operator $\mathcal{L}_{\xi}$ arising from the linearization of \eqref{GP_delta_intro} around $\phi_{\xi,v}$. Then, $\mathcal{L}_{\xi}$ admits a pair of eigenvalues that satisfy
    \begin{equation}
         \lambda^2 = -\frac{M_{\delta}''(0)}{P'_r(v)} g'(0)\xi + O(\xi^{3/2}),
    \label{eq: expansion_lambda_delta}
    \end{equation}
    where $$M_{\delta}''(0) = -2(1-v^2)^2 \qquad \text{and} \qquad g'(0) = \frac{(1-v^2)^2}{v^2}.$$
   In particular, the time-independent solutions $\phi_{\xi,v}(x)$ are spectrally unstable for $\xi>0$ small enough.
   \label{prop: delta}
\end{proposition}
The result in Proposition \ref{prop: delta} is in qualitative agreement with the one obtained in \cite[Figure 2(a)]{pham_3}, where the unstable eigenvalue of $\mathcal{L}_{\xi}$ is computed numerically. On the other hand, if $\xi < 0$, and hence $g<0$, our analysis is inconclusive. \\
Finally, we notice that the quantity $g'(0)$ in Proposition \ref{prop: delta} is singular at $v=0$, and that the zero velocity limit appears problematic. This issue can be solved by keeping in mind that $v$ and $g$ are the physical parameters of the problem; thus, zero velocity limit means that $v \to 0$, while the strength $g$ is kept constant. By \eqref{gamma_xi}, this implies that the parameter $\xi$ has to change with $v$ when we take the limit $v \to 0$. In this case, it is convenient to write the expression in \eqref{eq: expansion_lambda_delta} as a first order expansion in $g$, obtaining 
\begin{equation}
    \lambda^2 = -\frac{M_{\delta}''(0)}{P'_r(v)}g + O(g^{3/2}),
\end{equation}
and in the limit $v \to 0$, we obtain $\lambda^2 = \frac{1}{2}g$. This provides a justification of the formula appearing in \eqref{eq: formula_delta}.
\appendix
\section{Proof of Lemma \ref{lemma: cont_evans}}
\label{Section: appendix}
The goal of this Appendix is to extend the proof of Lemma \ref{lemma: cont_evans}, and to show that the Evans function $E(\lambda)$ of Definition \ref{evans_f} can be continued in an analytic way in a neighborhood of $\lambda =0$. In order to do this, we follow the approach in \cite[Proposition 2.7 and Lemma 2.10]{kapitula&sandstede}. \\
Consider the ODE system 
\begin{equation}
    Y'(\lambda,x) = M(\lambda,x) Y(\lambda,x)
    \label{appendix: ode}
\end{equation}
introduced in Section \ref{section: eigenvalues}, with the matrix $M(\lambda,x)$ given by \eqref{matrix_M}, and consider the asymptotic matrices $M_{\pm}(\lambda) = \lim_{x \to \pm \infty}M(\lambda,x)$, given in \eqref{matrices_pm}. The difference $M(\lambda,x) - M_{\pm}(\lambda)$ is independent of $\lambda$ and decays to zero exponentially fast as $|x| \to \infty$. We consider $\kappa>0$ such that 
\begin{equation}
    |M(\lambda,x) - M_{\pm}(\lambda)|e^{\pm5\kappa x} \leq C,
    \label{appendix: exponential_decay}
\end{equation}
for $C>0$, as in \cite[ Assumption 2.4]{kapitula&sandstede}. In particular, we can choose $\kappa = \frac{2}{5}\sqrt{1-v^2}$.\\
For $\delta >0$, we consider the set $$\Sigma_{\delta} := \{\lambda \in \mathbf{C}| \quad  \Re\lambda \leq 0 \quad \text{with} \quad |\lambda| < \delta\},$$ and we choose $\delta>0$ small enough so that the following conditions hold: the eigenvalues $\gamma_{2}$ and $\gamma_3$ of $M_{\pm}(\lambda)$ are analytic functions of $\lambda$ for any $\lambda \in \Sigma_{\delta}$; moreover, we have 
\begin{equation}
    \Re\gamma_2(\lambda) > -\frac{\kappa}{4} \qquad \text{and} \qquad \Re\gamma_3(\lambda) < \frac{\kappa}{4},
    \label{appendix: kappa_inequalities}
\end{equation} together with $\Re\gamma_1(\lambda) >0$ and $\Re\gamma_4(\lambda)<0$, for all $\lambda \in \Sigma_{\delta}$. Both conditions can be ensured by taking $\delta>0 $ small enough. In particular, the first condition on $\gamma_2$ and $\gamma_3$ can be obtained by using the fact that the eigenvalues of $M_{\pm}(\lambda)$ are analytic in the vicinity of $\lambda =0$  (see Lemma \ref{lemma: regular_roots}). Similarly, the inequalities above can be ensured by the continuity of the eigenvalues with respect to $\lambda$. Finally, we consider $r>0$ and the half-disk $D_r$ as in Lemma \ref{lemma: distinct_roots}, so that the eigenvalues of $M_{\pm}(\lambda)$ are analytic in $\lambda$ and distinct for any $\lambda \in D_r$. We denote $r':=\min\{r,\delta\}$ and we define $\Omega := \Sigma_{r'} \cup D_{r'}$. Then, the goal of this section is to construct an Evans function for $\lambda \in \Omega$ which is an analytic extension of $E(\lambda)$ from Definition \ref{evans_f}.\\
Let $Y_a$ and $Y_b$ be two solutions of the ODE system in \eqref{appendix: ode}. Consider the product $Y = Y_a \wedge Y_b \in \Lambda^2\mathbf{C}^4$, where $\wedge$ denotes the exterior product between vectors in $\mathbf{C}^4$, and $\Lambda^2\mathbf{C}^4$ denotes the second exterior power of $\mathbf{C}^4$ \cite{alexander, allen}. Then, $Y$ is a solution to 
\begin{equation}
    Y'(\lambda,x) = M^{(2)}(\lambda,x)Y(\lambda,x),
    \label{appendix: ode_system_exterior}
\end{equation}
where $M^{(2)}(\lambda,x)$ is defined as the linear derivation on $\Lambda^2\mathbf{C}^4$ induced by $M(\lambda,x)$, i.e. $M^{(2)}(\lambda,x)Y(\lambda,x) := M(\lambda,x)Y_a \wedge Y_b + Y_a \wedge M(\lambda,x) Y_b$ (see, for example, \cite[Section 2]{allen}). Next, consider the asymptotic system
\begin{equation}
    Y' = M^{(2)}_{\pm}(\lambda)Y,
\end{equation}
obtained in the limit $x \to \pm \infty$, and let's first focus on the case $x \to -\infty$. The eigenvalues of the matrix $M_{-}^{(2)}(\lambda)$ are the sums of any two eigenvalues of $M_{-}(\lambda)$. For $\lambda \in \Omega$, define $\alpha_-(\lambda):= \gamma_1(\lambda) + \gamma_2(\lambda)$. Then, $\alpha_{-}(\lambda)$ is analytic in $\lambda$ for $\lambda \in \Omega$, and, if $\Re \lambda >0$, we have that $\alpha_-(\lambda)$ corresponds to the eigenvalue of $M^{(2)}_-(\lambda)$ with the largest real part. In addition, $\alpha_-(\lambda)$ is a simple eigenvalue for $\lambda \in \Omega$ with $\Re\lambda >0$. We set 
\begin{equation}
    Z(\lambda, x) = Y(\lambda, x) e^{-\alpha_-(\lambda)x}.
    \label{appendix: def_Z}
\end{equation}
Then, $Z(\lambda, x)$ satisfies the ODE
\begin{equation}
    Z'(\lambda,x) = [M^{(2)}(\lambda,x) - \alpha_-(\lambda)\text{id}]\ Z(\lambda,x).
    \label{appendix: ode_Z_exterior}
\end{equation}
By setting
\begin{equation}
    x = \frac{1}{2\kappa} \ln \Bigl(\frac{1-\tau}{1+\tau}\Bigr),
\end{equation}
so that $\tau= \tanh(\kappa x)$, the system in \eqref{appendix: ode_Z_exterior} becomes the autonomous system
\begin{equation}
    Z'(\lambda,x) = [M^{(2)}(\lambda,\tau) - \alpha_-(\lambda)\text{id}]\ Z(\lambda,x), \qquad \tau' = \kappa(1-\tau^2),
    \label{appendix: ode_compact}
\end{equation}
where the derivative is always taken with respect to the variable $x$. The system in \eqref{appendix: ode_compact} can be shown to be $C^1$ on $\Lambda^2\mathbf{C}^4 \times [-1,1]$, thanks to the exponential decay \eqref{appendix: exponential_decay} (see \cite{alexander} and \cite{kapitula&sandstede}). For $\tau =-1$, this system reduces to 
\begin{equation*}
    Z' = [M_-^{(2)}(\lambda)-\alpha_-(\lambda)\text{id}]Z.
\end{equation*}
The critical points are the eigenvectors $\eta_-(\lambda)$ of $M^{(2)}_-(\lambda)$ associated to the eigenvalue $\alpha_-(\lambda)$, i.e. such that $[M^{(2)}_-(\lambda)-\alpha_-(\lambda)\text{id}]\eta_-(\lambda)=0$. For $\lambda \in \Omega$ with $\Re \lambda >0$, the eigenvector $\eta_-(\lambda)$ is analytic in $\lambda$, since the eigenvalue $\alpha_-(\lambda)$ of $M^{(2)}_-(\lambda)$ is simple. Moreover, it can be written as
\begin{equation}
    \eta_-(\lambda) = v_1^-(\lambda) \wedge v_2^-(\lambda),
    \label{appendix: def_eta}
\end{equation}
where $v_1^-(\lambda)$ and $v_2^-(\lambda)$ are the eigenvectors of $M_-(\lambda)$ corresponding to the eigenvalues $\gamma_1(\lambda)$ and $\gamma_2(\lambda)$, respectively (see Section \ref{subsection: eigenvectors}). If $\lambda \in \Omega$ with $\Re \lambda \leq 0$, then the eigenvalue $\alpha_-(\lambda)$ is not necessarily simple. However, the eigenvectors $v_1^-(\lambda)$ and $v_2^-(\lambda)$ are analytic in the vicinity of $\lambda =0$, and $\eta_-(\lambda)$ can be continued in an analytic way for $\lambda \in \Omega$ by \eqref{appendix: def_eta}, so that it continues to be an eigenvector of $M_-^{(2)}(\lambda)$ with eigenvalue $\alpha_-(\lambda)$. \\
Next, we linearize \eqref{appendix: ode_compact} at the critical point $(\eta_-(\lambda), -1)$. If $\lambda \in \Omega$ with $\Re \lambda >0$, the linearized system admits one unstable eigenvalue, equal to $2\kappa$, with associated eigenvector $(0,1)$. Suppose instead $\lambda \in \Omega$ with $\Re \lambda \leq 0$. We want to show that any eigenvalue of $M^{(2)}_-(\lambda) - \alpha_-(\lambda)\text{id}$ has real part smaller than $2\kappa$. In other words, that the unstable eigenvalue of the linearized system with largest real part is $2\kappa$, with the eigenvector still pointing in the $\tau$-direction.\\
For $\lambda \in \Omega$, let $\beta_-$ be the eigenvalue of $M^{(2)}_-(\lambda)$ with largest real part. If $\lambda \in \Omega$ with $\Re \lambda >0$, then the only eigenvalues of $M_-(\lambda)$ with positive real parts are $\gamma_1$ and $\gamma_2$. Thus we have $\beta_- = \alpha_-(\lambda)$. If instead $\lambda \in \Omega$ with $\Re\lambda \leq 0$, then $\Re\gamma_1(\lambda)>0$ and $\Re\gamma_4(\lambda)<0$, but $\gamma_2(\lambda)$ and $\gamma_3(\lambda)$ may have crossed the imaginary axis. However, using the assumption in \eqref{appendix: kappa_inequalities}, the sum of the real parts of $\gamma_2(\lambda)$ and $\gamma_3(\lambda)$ is at most $2\kappa/4$. Thus, we can write the following inequality
\begin{equation}
    \Re\beta_- - \Re\gamma_1(\lambda) < \frac{2\kappa}{4},
\end{equation}
for all $\lambda \in \Omega$ with $\Re \lambda \leq 0$. If we define $\beta_c^- := \beta_- - \alpha_-(\lambda)$, then we have $$\Re\beta_c^- <\kappa, \qquad \forall \lambda \in \Omega.$$ We conclude that the unstable eigenvalue of the linearized system with largest real part remains $2\kappa$, with eigenvector $(0,1)$. Thus, for $\lambda \in \Omega$, the point $(\eta_-(\lambda), -1)$ has a one-dimensional strong unstable manifold. Since the tangent vector to this manifold points in the $\tau$-direction, this can be written as a function of $\tau$, which we call $Z_-(\lambda, \tau)$, for $\tau$ near $-1$. By applying the flow of \eqref{appendix: ode_compact}, the solution $Z_-(\lambda, \tau)$ is defined for $\tau \in [-1,1)$. Moreover, by \cite[Lemma 2.2]{kapitula&sandstede} and the fact that $\eta_-(\lambda)$ is analytic in $\lambda$ for $\lambda \in \Omega$, we have that also $Z_-(\lambda,\tau)$ is analytic in $\lambda$ for $\lambda \in \Omega$. By equation \eqref{appendix: def_Z}, this defines a solution
\begin{equation*}
    Y_-(\lambda,x) = Z_-(\lambda,x)e^{\alpha_-(\lambda)x}
\end{equation*}
to \eqref{appendix: ode_system_exterior} such that $|Y_-(\lambda,x)| \to 0$ as $x \to -\infty$ exponentially fast. Moreover, $Y_-(\lambda,x)$ is analytic in $\lambda$ for $\lambda \in \Omega$. \\
Following the reasoning outlined above, we can construct a solution $Z_+(\lambda, \tau)$ as the strong unstable manifold of the point $(\eta_+(\lambda, +1)$, where $\eta_+(\lambda) := v_3^+(\lambda) + v^+_4(\lambda)$ is the eigenvector of $M^{(2)}_+(\lambda)$ associated with the eigenvalue $\alpha_+(\lambda):= \gamma_3(\lambda) +\gamma_4(\lambda)$. Here, $v_3^+(\lambda)$ and $v^+_4(\lambda)$ are the eigenvectors of $M_+(\lambda)$ associated with the eigenvalues $\gamma_3(\lambda)$ and $\gamma_4(\lambda)$ (see Section \ref{subsection: eigenvectors}). We obtain in this way a solution 
\begin{equation*}
    Y_+(\lambda,x) = Z_+(\lambda,x)e^{\alpha_+(\lambda)x}
\end{equation*}
to \eqref{appendix: ode_system_exterior} such that $|Y_+(\lambda,x)| \to 0$ as $x \to +\infty$ exponentially fast. Moreover, $Y_+(\lambda,x)$ is analytic in $\lambda$ for $\lambda \in \Omega$. Finally, we define the Evans function to be
\begin{equation*}
    E(\lambda) = Y_-(\lambda,x) \wedge Y_+(\lambda,x),
\end{equation*}
which takes values in $\Lambda^4\mathbf{C}^4 \cong \mathbf{C}$ and is analytic in $\lambda$ for $\lambda \in \Omega$. Moreover, for $\lambda \in \Omega$ with $\Re\lambda >0$, it coincides with the Evans function of Definition \ref{evans_f} \cite{kapitula&sandstede}.
\section*{Acknowledgments.}
\noindent The first author is partially supported by the PRIN 2022 Project 2022YXWSLR "Boundary analysis for dispersive and viscous fluids", by Istituto Nazionale di Alta Matematica through the GNAMPA Research Group, by the Italian Ministry of University and Research (MUR)
through the Excellence Department Project awarded to GSSI, CUP D13C22003740001. The second author is grateful to Diego Noja for interesting discussions on the instability of stationary solutions in Hamiltonian PDEs.
\bibliographystyle{siam}
\bibliography{biblio_1dsuperflow}

\end{document}